\documentclass[a4paper,11pt]{article}

\usepackage{amsmath}
\usepackage{amssymb}
\usepackage{amsthm}
\usepackage{amsfonts}
\usepackage{mathtools}
\usepackage{graphicx}
\usepackage{tikz,pgfplots}
\usepackage[utf8]{inputenc}
\usepackage{pgfplots}
\usepackage{tikz}
\pgfplotsset{compat=1.9}
\usepackage{subcaption}
\usepackage{caption}

\usepackage{enumerate}
\usepackage{titlesec}
\usepackage{enumitem}
\usepackage{mathtools}
\usepackage{listings}
\usepackage{bbm}
\usepackage{xcolor}
\usepackage{hyperref}
\usepackage{caption}
\usepackage{mathrsfs}
\usepackage{listings}
\usepackage{booktabs}
\usepackage{floatrow}
\newfloatcommand{capbtabbox}{table}[][\FBwidth]
\usepackage{blindtext}
\usepackage{tabularx}

\usepackage{colortbl}

\definecolor{asparagus}{rgb}{0.53, 0.66, 0.42}
\definecolor{ballblue}{rgb}{0.13, 0.67, 0.8}
\definecolor{cadmiumgreen}{rgb}{0.0, 0.42, 0.24}
\definecolor{cobalt}{rgb}{0.0, 0.28, 0.67}
\definecolor{darklavender}{rgb}{0.45, 0.31, 0.59}
\definecolor{green(pigment)}{rgb}{0.0, 0.65, 0.31}
\definecolor{blue(ncs)}{rgb}{0.0, 0.53, 0.74}
\definecolor{brandeisblue}{rgb}{0.0, 0.44, 1.0}
\definecolor{darkterracotta}{rgb}{0.8, 0.31, 0.36}
\definecolor{cobalt}{rgb}{0.0, 0.28, 0.67} 
\definecolor{ceruleanblue}{rgb}{0.16, 0.32, 0.75}
\definecolor{dodgerblue}{rgb}{0.12, 0.56, 1.0} 

\usepackage{geometry}
\newtheorem{theorem}{Theorem}[section]

\newtheorem{lemma}[theorem]{Lemma}

\newtheorem{definition}[theorem]{Definition}
\newtheorem{conclusion}[theorem]{Conclusion}

\renewcommand{\Re}{\mbox{\rm Re}}
\renewcommand{\Im}{\mbox{\rm Im}}

\newcommand{\bfu}{\boldsymbol{u}}
\newcommand{\bfv}{\boldsymbol{v}}
\newcommand{\bfw}{\boldsymbol{w}}
\newcommand{\bfz}{\boldsymbol{z}}

\newcommand{\bfy}{\boldsymbol{y}}

\makeatletter
\newcommand*{\rom}[1]{\expandafter\@slowromancap\romannumeral #1@}
\makeatother

\newcommand{\eps}{\varepsilon}

\newcommand{\dx}{\hspace{2pt}\mbox{d}x}
\newcommand{\dt}{\hspace{2pt}\mbox{d}t}

\newcommand{\bfL}{{\mathbf L}}
\newcommand{\bfH}{{\mathbf H}}
\newcommand{\D}{\mathcal{D}}

\newcommand{\calA}[1]{\mathcal{A}(#1)}
\newcommand{\calJ}[1]{\mathcal{J}(#1)}
\newcommand{\calJsigma}[1]{\mathcal{J}_{\sigma}(#1)}
\newcommand{\C}{\mathbb{C}}
\newcommand{\R}{\mathbb{R}}
\newcommand{\N}{\mathbb{N}}

\newcommand{\conjw}[1]{{\ensuremath{\overline{w}_{#1}}}} 
\newcommand{\conjv}[1]{{\ensuremath{\overline{v}_{#1}}}}
\newcommand{\ci}{\mathrm{i}} 
\renewcommand{\S}{\mathbb{S}}
\newcommand{\diff}{\mathop{}\!\mathrm{d}}

\newcommand{\dist}{\operatorname{dist}}

\usepackage{etoolbox}

\apptocmd{\thebibliography}{%
  \footnotesize
  \setlength{\itemsep}{0pt}%
  \setlength{\parsep}{0pt}%
}{}{}

\begin{document}
	\begin{center}
		{\LARGE 
				{\LARGE 
				Inverse Iterations for Phase-Invariant Nonlinear Eigenvalue Problems
				\renewcommand{\thefootnote}{\fnsymbol{footnote}}\setcounter{footnote}{0}
				\hspace{-3pt}\footnote{The authors acknowledge the support by the Deutsche Forschungsgemeinschaft (DFG, German Research Foundation) where P. Henning and L. Huynh are supported through the project grant 551527112.}}\\
	}
\end{center}

\begin{center}
	{\large Patrick Henning,\hspace{-2pt}\footnote[1]{\label{affiliation}Department of Mathematics, Ruhr-University Bochum, DE-44801 Bochum, Germany.\\ Email: \href{mailto:patrick.henning@rub.de}{patrick.henning@rub.de}, \href{mailto:Laura.Huynh@rub.de}{laura.huynh@rub.de}} \,\,
		Laura Huynh\textsuperscript{\ref{affiliation}} 
	}\\[2em]
\end{center}

\begin{abstract}
We study the local convergence of nonlinear inverse iterations for eigenvalue problems arising from variational energies that are invariant under global phase rotations. This symmetry introduces an inherent degeneracy of the target eigenvalue, preventing the direct application of existing convergence results that require a simple eigenvalue. Within an abstract variational framework, we investigate a shifted Jacobian-based inverse iteration, known as the $J$-method, and establish its local convergence to quasi-isolated ground states. Our analysis shows that the target eigenvalue of the $J$-operator has algebraic and geometric multiplicity exactly two. By introducing an auxiliary iteration that fixes the global phase, we prove local linear convergence of the iterates modulo phase rotations and derive sharp convergence rates in terms of the spectrum of the $J$-operator. In particular, these rates can be improved through spectral shifting. We apply the abstract theory to a generalized Gross-Pitaevskii equation that models spin--orbit coupled Bose--Einstein condensates, verify the relevant structural assumptions and numerically examine the quasi-isolation of the computed ground states. Numerical experiments demonstrate the effectiveness of the shifted $J$-method, particularly in regimes with slowly converging conventional inverse iterations.
\end{abstract}

\section{Introduction}
We consider nonlinear eigenvalue problems arising from the minimization of energy functionals $E$ under a fixed-normalization constraint. More precisely, our interest is in variational problems on complex-valued Hilbert spaces whose energies $E$ are invariant under global phase rotations $\bfu \mapsto \exp(\ci \omega) \bfu$ for rotation angles $\omega\in [-\pi,\pi)$. The corresponding constrained minimizers $\bfu$ solve, by the corresponding first-order optimality condition, a nonlinear eigenvalue problem of the form $E'(\bfu)=\lambda\mathcal{I}\bfu$. Since multiplication by a constant complex phase $\exp(\ci \omega)$ leaves both the energy and the normalization unchanged, every minimizer generates a family of equivalent solutions. This symmetry has important consequences for the convergence analysis of iterative eigenvalue solvers.

A prominent class of applications is provided by Gross--Pitaevskii models for Bose--Einstein condensates. Besides single-component condensates, these include systems with several interacting components, rotation, or spin--orbit coupling. The numerical computation of minimizers can be demanding, especially when competing interactions give rise to multiple local minima or nearly degenerate energy landscapes. A broad range of methods has therefore been developed, particularly for single-component condensates. One family exploits the underlying minimization problem through gradient flows, (Sobolev-)gradient descent, or Riemannian optimization; see, for example,~\cite{AHYY24,APS22,AGL24,ALT17,BaoCaiSurvey,BaoDu04,CDLX23,CLYZ24,CLYZ25,DaK10,DaP17,FengTang2025,FengTangWangIMA,HauckLiangPeterseim2024,HSW21,HeP20,PHMY242,Wang22}. A second family takes the nonlinear eigenvalue problem as its starting point, including self-consistent field iterations and generalizations of the inverse power method~\cite{AHP20,CKL21,CaLeBris00,JKM14,PH23,UpJaRu21}. Newton-type methods provide another alternative~\cite{APS23Newton,WuNewton17,XuXieNewton}. For a recent survey and a discussion of the connections between these approaches, we refer to~\cite{HJ25}.

Compared with the single-component case, numerical methods for multicomponent Bose--Einstein condensates have been investigated less extensively. The presence of several components introduces additional interactions and couplings, leading to a variety of nonlinear eigenvalue problems. Existing studies range from systems with interacting components~\cite{Bao04,101093imanumdraf046,HSY2025} to spinor condensates, where the components represent different internal spin states~
\cite{BaoCai2018SpinorReview,BCZ13,BaoLim2008,LiuCai21,CaiLiu21JCOMP,TiCaWuWe20,WuLiuCai25}. Further examples include two-component systems with Josephson coupling, which allows for transitions between the components~\cite{BaoCai2011,FengLiuTang2025}. Of particular interest for the present work are
spin--orbit coupled condensates, in which the internal spin states are coupled to the spatial motion of the atoms~\cite{BC15,HHP25}. Most numerical approaches
for these multicomponent problems are based on gradient flows or Riemannian optimization, while nonlinear inverse iterations have also been studied for the
spin--orbit coupled setting in~\cite{HHP25}. Rigorous convergence results are available for certain models and numerical schemes~\cite{FengLiuTang2025,101093imanumdraf046,HHP25}, but a general local convergence analysis of Jacobian-based inverse iterations that accounts for
global phase invariance is still missing.

In this work, we adopt the nonlinear eigenvalue perspective and study a generalized inverse iteration known as the \emph{$J$-method}. Originally introduced by Jarlebring \emph{et al.}~\cite{JKM14}, the method was subsequently analyzed in a general variational framework in~\cite{AHP20}, which includes, for example, real-valued single-component Gross--Pitaevskii problems. In contrast to inverse iterations that freeze the nonlinear coefficients at the preceding approximation, the $J$-method uses the Jacobian of a scaling-invariant representation of the nonlinear eigenvalue operator. This produces a Newton-like linearization while retaining the basic structure of inverse iteration; see also~\cite{HJ25}. 
Recently, the $J$-method was also employed to compute nonlinear stationary states of Bose--Einstein condensates confined to spherical shells, including solitonic and vortex states; see~\cite{WKE2026}.
This illustrates its practical relevance for complex-valued, phase-invariant nonlinear eigenvalue problems.

The local convergence analysis in~\cite[Theorem~3]{AHP20}, however, assumes that the eigenvalue of the $J$-operator associated with the desired minimizer (eigenfunction) is simple. This assumption fails for the phase-invariant problems considered here. Indeed, when the complex-valued function space is viewed as a real Hilbert space, global phase invariance implies that the eigenvalue $\lambda$ associated with a ground state $\bfu$ has at least two linearly independent eigenvectors, $\bfu$ and $\ci\bfu$. Thus the target eigenvalue of the $J$-operator is necessarily multiple, even for a single-component problem with complex-valued states, such as a rotating Gross--Pitaevskii model. Consequently, the existing simple-eigenvalue theory cannot be applied directly, despite the fact that the additional eigendirection merely reflects a physical symmetry.

We address this issue by developing a local convergence theory for the shifted $J$-method within an abstract framework of phase-invariant variational eigenvalue problems. Under a quasi-isolation assumption on the minimizer, we show that the target eigenvalue of the $J$-operator has algebraic and geometric multiplicity exactly two, corresponding precisely to the global phase symmetry. Using an auxiliary iteration that fixes this phase, we prove local linear convergence of the iterates modulo phase rotations, with sharp rates determined by the spectrum of the $J$-operator. We also obtain corresponding convergence results for the component densities and show how spectral shifting improves the local convergence rate. Our phase-fixing argument is inspired by~\cite{PHMY242}, where a related construction was used for Riemannian gradient descent. Finally, in Section~\ref{section:SO-application}, we apply the abstract theory to the spin--orbit coupled model discussed above, verify the relevant structural assumptions, numerically examine quasi-isolation of the computed minimizers and demonstrate the performance of the shifted $J$-method for this physical application.

\section{Mathematical setting and problem formulation}\label{section:2}

Before introducing the mathematical setting, we first fix the basic notation that will be used throughout the paper.
The spatial dimension is denoted by $d\in\{2,3\}$ and the number of components by $N\in \N$. To distinguish between single- and multi-component functions, we use boldface letters and adopt the identifications
\begin{align*}
\bfu = (u_1,\dots,u_N), \qquad \bfv = (v_1,\dots,v_N), \qquad \bfw = (w_1,\dots,w_N),
\end{align*}
for functions $u_j,v_j,w_j : \D \to \C$ ($j=1,\dots,N$) defined on the computational domain $\D \subset \R^d$.
This convention will be used throughout the work without further mention.

We denote by $L^2(\D,\C)$ and $H^1_0(\D,\C)$ the standard complex-valued Lebesgue and Sobolev spaces, respectively.
The corresponding spaces for multi-component functions are defined as
\begin{align*}
\bfL^2(\D) := L^2(\D,\C^N)
= \bigl\{\, \bfv = (v_1,\dots,v_N) \;\big|\; v_j \in L^2(\D,\C), \; j=1,\dots,N \,\bigr\},
\quad
\bfH^1_0(\D) := H^1_0(\D,\C^N).
\end{align*}
The space $\bfL^2(\D)$ is equipped with the real inner product
\begin{align*}
(\bfv,\bfw)_{\bfL^2(\D)}
:=
\Re \int_\D \bfv \cdot \overline{\bfw} \, \dx
= \Re \int_\D \sum_{j=1}^N v_j \conjw{j} \, \dx,
\end{align*}
and induced norm $\|\bfv\|_{\bfL^2(\D)} := (\bfv,\bfv)_{\bfL^2(\D)}^{1/2}$.
Similarly, we endow $\bfH^1_0(\D)$ with the inner product
\begin{align*}
(\bfv,\bfw)_{\bfH^1_0(\D)}
:= \Re \int_\D \sum_{j=1}^N \nabla v_j \cdot \nabla \conjw{j} \, \dx,
\end{align*}
and corresponding norm $\|\bfv\|_{\bfH^1_0(\D)} := (\bfv,\bfv)_{\bfH^1_0(\D)}^{1/2}$.
This choice of \emph{real} inner products ensures that $\bfL^2(\D)$ and $\bfH^1_0(\D)$ are real Hilbert spaces, which is convenient for the subsequent variational and optimization formulations.\\[0.4em]
For notational simplicity, we write $x=(x_1,x_2)$ when $\D\subset\R^2$ and $x=(x_1,x_2,x_3)$ when $\D\subset\R^3$.

\subsection{Minimization problem}
\label{sec:minimization-problem}

We consider an energy functional
$E:\bfH^1_0(\D)\to\mathbb{R}$ and seek its minimizers under a prescribed normalization constraint. More precisely, we minimize $E$ over the manifold
\begin{align*}
\mathbb{S} \,:=\, \{
\bfu=(u_1,\dots,u_N)\in\bfH^1_0(\D) \;\Big|\; \int_\D \sum_{j=1}^N |u_j|^2\,\dx=1 \}.
\end{align*}
The corresponding minimization problem reads
\begin{align}
\label{energy-minimization-problem}
\bfu \;\in\; \underset{\bfv\in\mathbb{S}}{\mathrm{arg\,min}}\; E(\bfv).
\end{align}
We refer to such minimizers as \emph{ground states}. In particular, they satisfy the normalization $\|\bfu\|_{\bfL^2(\D)}=1$. For the remainder of the general analysis, we impose the following assumptions.

\begin{enumerate}[label={(A\arabic*)}]
\item\label{A1}
$\D\subset\mathbb{R}^d$ is a bounded Lipschitz domain with $d\in\{2,3\}$.

\item\label{A2}
The energy functional $E:\bfH^1_0(\D)\to\mathbb{R}$ is three times continuously $\R$-Fr\'echet differentiable. Moreover, $E$ is weakly lower semicontinuous on $\bfH^1_0(\D)$, bounded from below on $\mathbb{S}$, and coercive on $\mathbb{S}$, i.e.,
\begin{align*}
E(\bfv)\to+\infty
\qquad\text{as }\|\bfv\|_{\bfH^1_0(\D)}\to\infty,
\quad \bfv\in\mathbb{S}.
\end{align*}

\item\label{A3}
The energy functional $E$ is invariant under global phase rotations, i.e.,
\begin{align}
\label{phase-shifts-E}
E(\exp(\ci\omega)\bfv)=E(\bfv)
\qquad\text{for all }\bfv\in\bfH^1_0(\D)
\text{ and }\omega\in\mathbb{R}.
\end{align}
\end{enumerate}

\medskip
\noindent
Under assumptions~\ref{A1} and~\ref{A2}, the minimization problem~\eqref{energy-minimization-problem} admits at least one ground state. Indeed, coercivity implies that every minimizing sequence is bounded in $\bfH^1_0(\D)$. By the compact embedding $\bfH^1_0(\D)\hookrightarrow\bfL^2(\D)$, a subsequence converges weakly in $\bfH^1_0(\D)$ and strongly in $\bfL^2(\D)$. Consequently, its limit belongs to $\mathbb{S}$, and the weak lower semicontinuity of $E$ yields the existence of a minimizer.

\medskip
\noindent
\textbf{Phase-rotation invariance.}
Assumption~\ref{A3} reflects an intrinsic symmetry of a broad class of Gross--Pitaevskii energy functionals, including standard single-component models, rotating condensates, and multi-component systems with spin-dependent or spin--orbit coupling; see, e.g.,~\cite{BaoCaiSurvey,BC15,BaoCai2018SpinorReview,PHMY242}. In these models, the energy is invariant under multiplication of all components by the same complex phase. In particular, if $\bfu\in\mathbb{S}$ is a ground state, then $\exp(\ci\omega)\bfu$ is also a ground state for every $\omega\in\mathbb{R}$, with
\begin{align*}
E(\exp(\ci\omega)\bfu)=E(\bfu).
\end{align*}
Hence, ground states cannot be unique as elements of $\mathbb{S}$, but may be unique up to a global phase rotation. Moreover, all phase-equivalent ground states have identical component densities,
\begin{align*}
|\exp(\ci\omega)u_j|^2=|u_j|^2
\qquad\text{for }j=1,\dots,N.
\end{align*}
Addressing this symmetry will play a central role in the subsequent spectral and convergence analysis of the $J$-method.

\subsection{First and second order conditions for minimizers}
\label{subsection-first-second-order-conditions}

Minimizers $\bfu \in \mathbb{S}$ of the energy functional $E$ under the normalization $\|\bfu\|_{\bfL^2(\D)}^2=1$ satisfy the usual first- and second-order optimality conditions. In the following, we first present these conditions in variational form and then introduce a suitable operator representation of the first variation.

\medskip
\noindent
\textbf{First-order (Euler--Lagrange) condition.}
Introducing a Lagrange multiplier $\lambda\in\mathbb{R}$ for the normalization constraint, we define the Lagrangian
$$
\mathscr{L}(\bfu,\lambda) := E(\bfu) - \frac{\lambda}{2} \Big( (\bfu,\bfu)_{\bfL^2(\D)} - 1 \Big).
$$
A necessary condition for $\bfu \in \mathbb{S}$ to be a critical point of $E$ on $\mathbb{S}$ is that the first variation vanishes:
\begin{align}
\label{necessary-cond-min}
\langle E'(\bfu), \bfv \rangle = \lambda\, (\bfu, \bfv)_{\bfL^2(\D)}
\qquad \text{for all } \bfv \in \bfH^1_0(\D).
\end{align}
This is a nonlinear eigenvalue problem for the eigenpair $(\bfu,\lambda)\in\mathbb{S}\times\mathbb{R}$, where $\lambda$ is the Lagrange multiplier associated with the normalization constraint.

\medskip
\noindent
\textbf{Second-order condition.}
Let the tangent space to $\mathbb{S}$ at $\bfu$ be
$$
T_{\bfu}\mathbb{S} := \{ \bfv \in \bfH^1_0(\D) \;|\; (\bfu,\bfv)_{\bfL^2(\D)} = 0 \}.
$$
A necessary condition for $\bfu$ to be a (local) minimizer is that the second variation of the Lagrangian is nonnegative on the tangent space:
$$
\langle E''(\bfu) \bfv, \bfv \rangle - \lambda (\bfv,\bfv)_{\bfL^2(\D)} \ge 0
\qquad \text{for all } \bfv \in T_{\bfu}\mathbb{S}.
$$
In other words, the constrained Hessian is positive semi-definite on $T_{\bfu}\mathbb{S}$.

\medskip
\noindent
Recall that the energy functional $E$ is invariant under global phase rotations by assumption~\ref{A3}. Differentiating this invariance  \eqref{phase-shifts-E} with respect to $\bfu$ yields the corresponding phase equivariance of $E'$, i.e., $\langle E'(\exp(\ci\omega)\bfu),\bfv\rangle = \langle E'(\bfu),\exp(-\ci\omega)\bfv\rangle$ for all $\bfv\in\bfH^1_0(\D)$. Differentiating once more with respect to the phase parameter at $\omega=0$ and using the first-order condition~\eqref{necessary-cond-min}, we obtain
\begin{align}\label{secE-iu-ev}
\langle E''(\bfu)\, [\ci \bfu], \bfv \rangle - \lambda (\ci \bfu, \bfv)_{\bfL^2(\D)} = 0
\qquad \text{for all } \bfv \in \bfH^1_0(\D).
\end{align}
Consequently, the direction $\ci \bfu \in T_{\bfu}\mathbb{S}$ is always a neutral direction for the second variation. This is the variational manifestation of the $\mathrm{U}(1)$ phase symmetry.

In the light of the above discussion, a local minimizer $\bfu \in \mathbb{S}$ of $E$ is called
\emph{quasi-isolated} (cf.~\cite{HY25MathComp,PHMY242}) if it satisfies the following condition:
\begin{enumerate}[resume,label={(A\arabic*)}]
\item\label{A4} The constrained Hessian of $E$ at $\bfu$ is strictly positive
on the subspace orthogonal to both $\bfu$ and $\ci\bfu$, i.e.,
\begin{align*}
\langle E''(\bfu)\bfv,\bfv\rangle - \lambda(\bfv,\bfv)_{\bfL^2(\D)} > 0
\qquad \text{for all } \bfv \in (T_{\bfu}\mathbb{S} \cap T_{\ci\bfu}\mathbb{S}) \setminus\{\mathbf{0}\}.
\end{align*}
\end{enumerate}
This means that the phase direction $\ci\bfu$ is the only neutral direction on $T_{\bfu}\mathbb{S}$ that keeps the energy unchanged. In particular, if a suitable G\r{a}rding inequality holds (cf.~\ref{A5}) assumption \ref{A4} ensures that $\bfu$ is locally isolated as a minimizer on $\mathbb{S}$ up to the global phase symmetry.

\section{Nonlinear inverse iterations}
\label{section:inverse-iterations}

In this section, we introduce nonlinear inverse iterations for the computation of ground states of the energy functional $E$. Recall from Section~\ref{subsection-first-second-order-conditions} that every ground state $\bfu\in\mathbb{S}$ satisfies the first-order optimality condition~\eqref{necessary-cond-min}, which can be interpreted as a nonlinear eigenvalue problem
\begin{align}
\label{eq:nonlinear-eigenvalue-problem}
E'(\bfu)=\lambda\mathcal{I}\bfu,
\end{align}
where $\mathcal{I}:\bfH^1_0(\D)\to\bfH^{-1}(\D)$ denotes the identification $\mathcal{I}\bfu:=(\bfu,\cdot)_{\bfL^2(\D)}$.

Our goal is to solve this eigenvalue problem by generalizing the classical inverse power method to the nonlinear setting. To this end, we seek suitable linearized operators that allow us to approximate an eigenfunction through a sequence of linear problems, followed by normalization. We consider two approaches: the canonical nonlinear inverse iteration, referred to as the $A$-method, and the $J$-method, which is based on a different linearization of the nonlinear eigenvalue problem and allows for acceleration through spectral shifting.

\subsection{The $A$-method}
\label{subsection:A-method}

We first consider the canonical nonlinear inverse iteration, which we refer to as the \emph{$A$-method}. To make the nonlinear eigenvalue problem~\eqref{necessary-cond-min} accessible to inverse iteration methods, we seek a family of linear operators
$\calA{\bfu}:\bfH^1_0(\D)\to\bfH^{-1}(\D)$
such that
\begin{align}
\label{energy-A-identity}
\calA{\bfu}\bfu=E'(\bfu) \qquad\text{for all }\bfu\in\mathbb{S}.
\end{align}
Such an operator family can be constructed directly from the energy functional $E$. Since the phase invariance~\eqref{phase-shifts-E} implies $E'(\mathbf{0})=0$, we can define, for $\bfu\in\bfH^1_0(\D)\setminus\{\mathbf{0}\}$,
\begin{align}
\label{def-Lu}
\langle\calA{\bfu}\bfv,\bfw\rangle
\,:=\, \int_0^1 \langle E''\big(\tfrac{t\,\bfu}{\|\bfu\|_{\bfL^2(\D)}}\big) \bfv,\bfw \rangle\,\dt
\end{align}
for all $\bfv,\bfw\in\bfH^1_0(\D)$. By the fundamental theorem of calculus, this yields
\begin{align}
\label{normalized-energy-A}
\calA{\bfu}\bfu \,=\, \|\bfu\|_{\bfL^2(\D)} \,  E'\big(\tfrac{\bfu}{\|\bfu\|_{\bfL^2(\D)}}\big).
\end{align}
In particular, $\calA{\bfu}\bfu=E'(\bfu)$ for all $\bfu\in\mathbb{S}$, as required.
Moreover, since the phase invariance of $E$ implies that $E$ is even, i.e., $E(-\bfu)=E(\bfu)$, differentiating twice yields $E''(-\bfu)=E''(\bfu)$. Consequently, $\calA{\bfu}$ is \emph{real-scaling invariant}, i.e.,
\begin{align}
\label{scaling-invariance-A}
\calA{\alpha\bfu}=\calA{\bfu}
\qquad\text{for all }\alpha\in\mathbb{R}\setminus\{0\}.
\end{align}
With this construction, the nonlinear eigenvalue problem~\eqref{necessary-cond-min} can be equivalently rewritten as
\begin{align}
\label{eq:eigenvalue-A}
\calA{\bfu}\bfu=\lambda\mathcal{I}\bfu.
\end{align}
The operator representation~\eqref{def-Lu} is not unique and does not necessarily coincide with the canonical linearized operators used in Gross--Pitaevskii models. Depending on the underlying energy functional, alternative operator families satisfying~\eqref{energy-A-identity} and~\eqref{scaling-invariance-A} may be preferable, in particular when their (uniform) coercivity of $\calA{\bfu}$ can be established more easily.

For the following inverse iterations, we assume that the chosen operator $\calA{\bfu}$ is bounded, symmetric and coercive for every $\bfu\in\mathbb{S}$, which is valid for typical realizations used in practice, cf. \cite{PH23,HJ25,HHP25}. Hence, the Lax--Milgram theorem guarantees the existence of a bounded inverse
$$
\calA{\bfu}^{-1}:\bfH^{-1}(\D)\to\bfH^1_0(\D).
$$
For a given $\bfw\in\bfH^1_0(\D)$, the function \,$\bfz:=\calA{\bfu}^{-1}\mathcal{I}\bfw$\, 
is therefore the unique solution to
\begin{align*}
\langle\calA{\bfu}\bfz,\bfv\rangle \,=\, (\bfw,\bfv)_{\bfL^2(\D)}
\qquad\text{for all }\bfv\in\bfH^1_0(\D).
\end{align*}
This allows us to formulate the nonlinear inverse iteration as follows.
\begin{definition}[$A$-method]
Let $\bfu^0\in\mathbb{S}$ denote an initial guess. Then the iterates of the $A$-inverse iteration are defined for $n\ge0$ by
\begin{align}
\label{A-iteration}
\bfu^{n+1}
:=
\frac{\hspace{-24pt}\calA{\bfu^n}^{-1}\mathcal{I}\bfu^n}
{\|\calA{\bfu^n}^{-1}\mathcal{I}\bfu^n\|_{\bfL^2(\D)}}.
\end{align}
\end{definition}
Each iteration step~\eqref{A-iteration} requires the solution of one linear problem involving $\calA{\bfu^n}$. By construction, the normalization ensures $\bfu^n\in\mathbb{S}$ for all $n\in\mathbb{N}$. Consequently, the normalization factors appearing in~\eqref{def-Lu} are constant (equal to one) along the iteration.

To control the step size and ensure energy decay, the $A$-method can be equipped with a damping mechanism. In this variant, the next iterate is obtained as a suitable combination of the previous approximation and the standard inverse iterate, followed by normalization.

\begin{definition}[Damped $A$-method]
For an initial guess $\bfu^0\in\mathbb{S}$ and a sequence of positive damping parameters $(\tau_n)_{n\in\mathbb{N}}\subset(0,2)$, the iterates of the damped $A$-inverse iteration are defined for $n\ge0$ by
\begin{align}
\label{damped-A-iteration}
\bfu^{n+1} := \frac{\hspace{-24pt}(1-\tau_n)\bfu^n
+\tau_n\,\gamma_n\,\calA{\bfu^n}^{-1}\mathcal{I}\bfu^n}
{\big\|(1-\tau_n)\bfu^n
+\tau_n\,\gamma_n\,\calA{\bfu^n}^{-1}\mathcal{I}\bfu^n\big\|_{\bfL^2(\D)}},
\end{align}
where
$$
\gamma_n := \big(\calA{\bfu^n}^{-1}\mathcal{I}\bfu^n,\bfu^n\big)_{\bfL^2(\D)}^{-1}.
$$
\end{definition}
The factor $\gamma_n$ is chosen such that the update direction
$$
\mathbf{d}^n
:=
-\bfu^n+\gamma_n\,\calA{\bfu^n}^{-1}\mathcal{I}\bfu^n
$$
is tangent to the constraint manifold $\mathbb{S}$, i.e., $\mathbf{d}^n\in T_{\bfu^n}\mathbb{S}$. Moreover, the definition of $\mathbf{d}^n$ and the coercivity of $\calA{\bfu^n}$ imply
\begin{align*}
\langle E'(\bfu^n),\mathbf{d}^n\rangle \,=\, -\langle\calA{\bfu^n}\mathbf{d}^n,\mathbf{d}^n\rangle
\le0,
\end{align*}
with strict inequality unless $\mathbf{d}^n=\mathbf{0}$. Hence, $\mathbf{d}^n$ is a descent direction for the energy. The iteration~\eqref{damped-A-iteration} can therefore be equivalently written as
\begin{align*}
\bfu^{n+1} = \frac{\bfu^n+\tau_n\,\mathbf{d}^n}{\|\bfu^n+\tau_n\,\mathbf{d}^n\|_{\bfL^2(\D)}}.
\end{align*}
This corresponds to a Riemannian gradient descent method on $\mathbb{S}$ with an energy-adaptive metric induced by $\calA{\bfu^n}$, cf.~\cite{HeP20,HHP25}. For a general discussion of the connections between nonlinear inverse iterations, Sobolev gradients and Riemannian optimization, we refer to the review~\cite{HJ25}.

\medskip
\noindent
\textbf{Advantages and limitations.}
A major advantage of the damped $A$-method is its close connection to Riemannian gradient descent. By choosing sufficiently small step sizes or employing a suitable line search, one can ensure that the energy decreases in every iteration. Under appropriate additional assumptions, this energy dissipation can be used to establish global convergence results towards stationary points of the constrained minimization problem, at least along subsequences; see~\cite{HeP20,PHMY242,HHP25}. Such results have been obtained for several Gross--Pitaevskii models, making the damped $A$-method a robust approach for obtaining initial approximations of ground states.

However, the method may suffer from slow local convergence, particularly when the energy landscape is weakly curved in certain tangential directions near a minimizer. In such situations, the resulting gradient dynamics can require many iterations to resolve the remaining error. The local convergence rates are closely connected to the spectral properties of the linearized problem and the constrained Hessian; see~\cite{JKM14,PH23,PHMY242,FengTang2025}.

A further limitation is that the $A$-method does not generally benefit from spectral shifting in the same way as classical linear inverse iteration. In particular, for Gross--Pitaevskii problems, shifting the $A$-operator towards the target eigenvalue may deteriorate the convergence rate or even lead to divergence, rather than providing the desired acceleration; see~\cite{AHP20,PH23,HJ25}. This motivates the alternative linearization underlying the $J$-method, which we introduce next.

\subsection{The $J$-method}\label{subsection:J-method}
Next, we introduce the $J$-method as an inverse iteration for nonlinear eigenvalue problems arising from constrained energy minimization. We formulate the method with spectral shifting and present its local convergence properties, taking into account the invariance under global phase rotations.
\subsubsection{The $J$-operator}

Let $\calA{\bfu}$ denote the operator from Section \ref{subsection:A-method}. Following the approach of~\cite{JKM14,AHP20}, we introduce the $J$-operator as the real Fr\'echet derivative of the nonlinear mapping $\bfu\mapsto\calA{\bfu}\bfu$, i.e.,
\begin{align}
\label{def-J-op}
\calJ{\bfu} := \frac{\mathrm{d}}{\mathrm{d}\bfu}
\big[\calA{\bfu}\bfu\big] : \bfH^1_0(\D)\to\bfH^{-1}(\D),
\qquad \bfu\neq\mathbf{0}.
\end{align}
By the real-scaling invariance \eqref{scaling-invariance-A} of the $A$-operator, we immediately
obtain
\begin{align}
\label{identity-A-J}
\calJ{\bfu}\bfu=\calA{\bfu}\bfu.
\end{align}
Indeed, for any $\bfu\neq\mathbf{0}$, we have
\begin{align}
\label{eq-Au-Ju-id}
\calJ{\bfu}\bfu \,=\,
\lim_{\varepsilon\to0}
\frac{1}{\varepsilon}
\Big( \calA{(1+\varepsilon)\bfu} (1+\varepsilon)\bfu - \calA{\bfu}\bfu \Big)
\,=\,\calA{\bfu}\bfu.
\end{align}
Consequently, the nonlinear eigenvalue problems \eqref{necessary-cond-min} and \eqref{eq:eigenvalue-A} can be equivalently written as
\begin{align}
\label{eq:eigenvalue_J}
\calJ{\bfu}\bfu=\lambda\mathcal{I}\bfu, \qquad \bfu\in\mathbb{S}.
\end{align}
\medskip
\noindent
\textbf{Connection to the canonical Hessian $E''(\bfu)$.}\\
An important property of the $J$-operator (compared to the $A$-operator) is its close relationship to $ E''(\bfu)$. Indeed, differentiating the normalized energy identity
\eqref{normalized-energy-A} at $\bfu\in\mathbb{S}$ yields
\begin{align}
\label{eq:J}
\langle\calJ{\bfu}\bfv,\bfw\rangle
\,=\, \langle E''(\bfu)\bfv,\bfw\rangle \,+\, (\bfu,\bfv)_{\bfL^2(\D)} \,\big( \langle E'(\bfu),\bfw\rangle - \langle E''(\bfu)\bfu,\bfw\rangle \big)
\end{align}
for all $\bfv,\bfw\in\bfH^1_0(\D)$. In particular, the additional term vanishes whenever $\bfv$ belongs to the tangent space $T_{\bfu}\mathbb{S}$. Hence,
\begin{align}
\label{identity-Ju-Eprimeprime}
\langle\calJ{\bfu}\bfv,\bfw\rangle
=
\langle E''(\bfu)\bfv,\bfw\rangle
\qquad\text{for all }
\bfv\in T_{\bfu}\mathbb{S}
\text{ and }
\bfw\in\bfH^1_0(\D).
\end{align}
Although $\calJ{\bfu}$ is generally nonsymmetric, its action on tangent directions therefore agrees
with the Hessian of the energy. In particular, the corresponding bilinear forms coincide on $T_{\bfu}\mathbb{S}$:
\begin{align*}
\langle\calJ{\bfu}\bfv,\bfv\rangle \,=\,
\langle E''(\bfu)\bfv,\bfv\rangle \qquad
\text{for all }\bfv\in T_{\bfu}\mathbb{S}.
\end{align*}
This identity is crucial for the following spectral analysis of the $J$-operator at ground states.

\subsubsection{$J$-inverse iteration with spectral shifting}

Based on the $J$-formulation \eqref{eq:eigenvalue_J} of the nonlinear eigenvalue problem, we now introduce a corresponding inverse iteration. In contrast to the $A$-inverse iteration \eqref{A-iteration}, which does not benefit from spectral shifting for problems of Gross--Pitaevskii type (cf.~\cite{PH23,HJ25}), spectral shifting can substantially accelerate the local convergence of the $J$-method.

For a shift $\sigma\in\R$, we introduce the
shifted $J$-operator
\begin{align*}
\calJsigma{\bfu} := \calJ{\bfu}-\sigma\mathcal{I} : \bfH^1_0(\D)\to\bfH^{-1}(\D),
\end{align*}
where $\mathcal{I}\bfv=(\bfv,\cdot)_{\bfL^2(\D)}$. To formulate the inverse iteration, we require that $\calJsigma{\bfu}$ admits a bounded inverse. We will later show that this is the case for suitably chosen shifts $\sigma$. Assuming invertibility, we define
$$
\bfz \,:=\, \calJsigma{\bfu}^{-1}\mathcal{I}\bfw \in\bfH^1_0(\D)
$$
for a given $\bfw\in\bfH^1_0(\D)$. Equivalently, $\bfz$ is the unique solution of
\begin{align*}
\langle\calJsigma{\bfu}\bfz,\bfv\rangle \,=\, (\bfw,\bfv)_{\bfL^2(\D)}
\qquad
\text{for all }\bfv\in\bfH^1_0(\D).
\end{align*}
With this preparation, we define the shifted
$J$-method as follows.

\begin{definition}[$J$-method with shifting]
Let $\bfu^0\in\mathbb{S}$ denote an initial guess. Then the iterates of the $J$-inverse iteration
with shift $\sigma\in\R$ are defined for $n\geq0$ by
\begin{align}
\label{J-iteration-shift}
\bfu^{n+1} \,:=\,
\operatorname{sgn}(\lambda-\sigma)
\frac{\calJsigma{\bfu^n}^{-1}\mathcal{I}\bfu^n}{\|\calJsigma{\bfu^n}^{-1}\mathcal{I}\bfu^n\|_{\bfL^2(\D)}},
\end{align}
where $\lambda$ denotes the target eigenvalue and the factor $\operatorname{sgn}(\lambda-\sigma)$ prevents sign oscillations of the iterates.
\end{definition}

In practice, the factor $\operatorname{sgn}(\lambda-\sigma)$ is either omitted, accepting possible sign changes between consecutive iterates, or heuristically set to $-1$ if sign oscillations are observed and to $1$ otherwise. Since multiplication by $-1$ is a global phase rotation, such sign changes do not
affect the particle densities or the eigenvalue approximations. For the theoretical convergence
analysis, however, the sign factor is kept to avoid additional case distinctions.

The shift $\sigma$ should be chosen close to, but different from, the target eigenvalue $\lambda$. The resulting convergence rate will be quantified in the next subsection. A suitable approximation of $\lambda$ can be obtained by performing a few iterations of the damped $A$-method \eqref{damped-A-iteration}.

\subsubsection{Local convergence rate of the shifted $J$-method}
The shifted $J$-method is designed to converge locally to an eigenfunction of the nonlinear eigenvalue problem~\eqref{eq:nonlinear-eigenvalue-problem}. A natural strategy is therefore to first employ the damped $A$-method \eqref{damped-A-iteration} to obtain a sufficiently accurate initial approximation and then switch to the shifted $J$-method. The transition can, for instance, be based on the residual or the observed energy reduction.

Before stating the local convergence result, we introduce a metric that accounts for the
phase invariance of the minimization problem.

\medskip
\noindent
\textbf{Quotient metric modulo global phase.}\\
Due to the phase invariance~\eqref{phase-shifts-E}, ground states are determined only up to multiplication by a global complex phase. Moreover, all component densities remain unchanged under such rotations. It is therefore natural to measure errors modulo global phase rotations. Let 
$S^1 := \{\,\exp(\ci\omega)\;|\;\omega\in\R\,\}$ denote the group of complex scalars of modulus one.
For $\bfu,\bfv\in\mathbb{S}$, we define
\begin{align}
\label{quotient-metric}
\dist_{S^1}^{\bfH^1}(\bfu,\bfv)
:=
\min_{\omega\in\R}
\|\exp(\ci\omega)\bfu-\bfv\|_{\bfH^1(\D)}.
\end{align}
This induces a well-defined metric on the quotient manifold $\mathbb{S}/S^1$, i.e., on normalized states identified up to global phase rotations.

The minimizing phase can be characterized explicitly. Expanding the squared $\bfH^1(\D)$ distance shows that $\omega_*$ satisfies
\begin{align*}
\omega_*
\,\in\, \arg\max_{\omega\in\R} (\exp(\ci\omega)\bfu,\bfv)_{\bfH^1(\D)}
\,=\, \arg\max_{\omega\in\R} \,\Re\int_\D \sum_{j=1}^{N} 
\big( u_j\conjv{j} + \nabla u_j\cdot\nabla\conjv{j} \big) \exp(\ci\omega)\,\dx.
\end{align*}
Hence, $\omega_*$ aligns $\bfu$ with $\bfv$ by maximizing their real-valued $\bfH^1$ correlation. Equivalently,
\begin{align*}
\dist_{S^1}^{\bfH^1}(\bfu,\bfv)^2
\,=\,
\|\bfu\|_{\bfH^1(\D)}^2 + \|\bfv\|_{\bfH^1(\D)}^2 - 2\,\max_{\omega\in\R} (\exp(\ci\omega)\bfu,\bfv)_{\bfH^1(\D)}.
\end{align*}
\medskip
\noindent
\textbf{Convergence result.}
We now formulate the local convergence result for the shifted $J$-method in the abstract setting of phase-invariant energy minimization problems. For this, we impose an additional assumption on the $J$-operator that guarantees the required spectral properties.

\begin{enumerate}[resume,label={(A\arabic*)}]
\item\label{A5}
At the ground-state eigenpair $(\bfu,\lambda)\in\mathbb{S}\times\mathbb{R}$ under consideration, the operator $\calJ{\bfu}$ satisfies a G\r{a}rding inequality: there exist constants $\alpha>0$ and $\beta\geq0$ such that
\begin{align}
\label{assumption-J-Garding}
\langle\calJ{\bfu}\bfv,\bfv\rangle 
\geq \alpha\|\bfv\|_{\bfH^1_0(\D)}^2 - \beta\|\bfv\|_{\bfL^2(\D)}^2
\qquad\text{for all }\bfv\in\bfH^1_0(\D).
\end{align}
\end{enumerate}

Together with the compact embedding $\bfH^1_0(\D)\hookrightarrow\bfL^2(\D)$, assumption~\ref{A5} provides the Fredholm and compactness properties needed for the spectral analysis of $\calJ{\bfu}$. In particular, under the quasi-isolation assumption~\ref{A4}, we will establish that the generalized spectrum at the ground state is discrete and real.

\begin{theorem}[Local convergence rate]
\label{theo:local}
Assume \ref{A1}--\ref{A3}, and let $(\bfu,\lambda)\in\mathbb{S}\times\mathbb{R}$ be a ground-state eigenpair of~\eqref{eq:nonlinear-eigenvalue-problem} that is quasi-isolated in the sense of~\ref{A4} and satisfies~\ref{A5}. Let
\begin{align*}
\Lambda(\bfu) \,:=\, \big\{ \mu\in\mathbb{R}: \exists\,\mathbf{0}\neq\bfv\in\bfH^1_0(\D)
\text{ with } \calJ{\bfu}\bfv=\mu\mathcal{I}\bfv \big\}
\end{align*}
denote the discrete set of real eigenvalues of $\calJ{\bfu}$. For a shift $\sigma\in\mathbb{R}\setminus\{\lambda\}$ satisfying
\begin{align}
\label{eq:shift-condition}
|\lambda-\sigma|
\,<\, \min_{\substack{\mu\in\Lambda(\bfu)\\\mu\neq\lambda}} |\mu-\sigma|
\,=:\, |\mu_0-\sigma|,
\end{align}
we consider the iterates of the shifted $J$-method~\eqref{J-iteration-shift}.

Then, for every $\varepsilon>0$, there exist a neighbourhood $U_\varepsilon\subset\bfH^1_0(\D)$ of $\bfu$ and a constant $C_\varepsilon>0$ such that, for all $\bfu^0\in U_\varepsilon\cap\mathbb{S}$ and all $n\geq1$,
\begin{align}
\label{eq:local-convergence-rate}
\dist_{S^1}^{\bfH^1}(\bfu,\bfu^n)
\,\,\le\,\, C_\varepsilon
\left( \frac{|\lambda-\sigma|}{|\mu_0-\sigma|} + \varepsilon \right)^n \|\bfu^0-\bfu\|_{\bfH^1(\D)}.
\end{align}
In particular, the iterates converge locally linearly to $\bfu$ in the quotient metric modulo global phase rotations. Moreover, the component densities converge at the same linear rate:
\begin{align}
\label{eq:density-convergence-rate}
\big\||\bfu^n|^2-|\bfu|^2\big\|_{\bfL^2(\D)} \,\,\le\,\,  
C_\varepsilon \left( \frac{|\lambda-\sigma|}{|\mu_0-\sigma|} + \varepsilon \right)^n
\|\bfu^0-\bfu\|_{\bfH^1(\D)}.
\end{align}
\end{theorem}

The proof splits into several parts with the core proof given in Section~\ref{sec:proof-loc-convergence}.

\section{Properties of the $J$-operator}\label{section:5}
In this section, we investigate the spectral properties of the $J$-operator that are essential for our convergence analysis. We first establish that the target eigenvalue $\lambda$ has algebraic and geometric multiplicity exactly two under the quasi-isolation assumption \ref{A4}. In particular, this result identifies the phase-induced degeneracy that prevents a direct application of convergence arguments based on simple eigenvalues.

\begin{lemma}[Eigenspace to $\lambda$]\label{lem:J-mult-2-nonsym}
Assume \ref{A1}--\ref{A3}, and let $(\bfu,\lambda)\in\mathbb{S}\times\mathbb{R}$ be a ground-state eigenpair satisfying the quasi-isolation assumption~\ref{A4} and the G\r{a}rding assumption~\ref{A5}. Then it holds
\begin{align*}
\ker\!\big(\mathcal{J}(\bfu)-\lambda\mathcal{I}\big)
=\operatorname{span}\{\bfu,\ci\bfu\} \qquad \mbox{(i.e., the geometric multiplicity of $\lambda$ is $2$)}
\end{align*}
and there are no Jordan chains at $\lambda$;\ hence the algebraic multiplicity of $\lambda$ is also $2$.
\end{lemma}

\begin{proof}
We already verified that the scaling invariance of $\calA{\bfu}$ implies 
\begin{align*}
\mathcal{J}(\bfu)\bfu \overset{\eqref{eq-Au-Ju-id}}{=} \mathcal{A}(\bfu)\bfu=\lambda\mathcal{I}\bfu,
\end{align*}
i.e., $\bfu$ is an eigenfunction of $\mathcal{J}(\bfu)$ to the eigenvalue $\lambda$. Furthermore, by the phase invariance~\eqref{phase-shifts-E} and the normalized energy identity~\eqref{normalized-energy-A}, we also have
\begin{align*}
\mathcal{J}(\bfu)[\ci\bfu]
=\frac{\mathrm d}{\mathrm d\theta}\Big|_{\theta=0}
\calA{\exp(\ci\theta)\bfu}[\exp(\ci\theta)\bfu]
=\ci\,\calA{\bfu}\bfu=\lambda\,\mathcal{I}[\ci\bfu],
\end{align*}
i.e., $\ci \bfu$ is a second eigenfunction of $\mathcal{J}(\bfu)$ to the eigenvalue $\lambda$. 
Next, recall \ref{A4}, i.e.,
\begin{align}
\label{A4-new}
\langle (E''(\bfu)-\lambda\mathcal{I})\bfv,\bfv\rangle>0
\quad\text{for all }\bfv\in T_{\bfu}\mathbb{S}\cap T_{\ci\bfu}\mathbb{S}\setminus\{\mathbf0\}.
\end{align}
To show that the geometric multiplicity of $\lambda$ is $2$, we let $\bfv \in \bfH^1_0(\D)$ be any other eigenfunction with $(\mathcal{J}(\bfu)-\lambda\mathcal{I})\bfv=\mathbf0$. We can decompose it uniquely as
\begin{align*}
\bfv = a\,\bfu + b\,\ci\bfu + \bfw,
\qquad \bfw\in T_{\bfu}\mathbb{S}\cap T_{\ci\bfu}\mathbb{S}, \qquad a,b \in \R.
\end{align*}
Since $\bfu$ and $\ci \bfu$ are eigenfunctions to $\lambda$, we obtain
\begin{align*}
\mathbf0 \,\,=\,\, (\mathcal{J}(\bfu)-\lambda\mathcal{I})\bfv =  (\mathcal{J}(\bfu)-\lambda\mathcal{I})\bfw.
\end{align*}
Applying \eqref{identity-Ju-Eprimeprime} to $\bfw\in T_{\bfu}\mathbb{S}$, we obtain 
\begin{align*}
0 \,\,\, = \,\,\,  \langle (\mathcal{J}(\bfu)-\lambda\mathcal{I}) \bfw, \bfw \rangle 
	\,\,\, \overset{\eqref{identity-Ju-Eprimeprime}}{=} \,\,\, 
\langle (E''(\bfu)-\lambda\mathcal{I}) \bfw, \bfw \rangle \,\,\, \overset{\eqref{A4-new}}{>}\,\,\, 0 \quad\text{if }\bfw\neq \mathbf0,
\end{align*}
a contradiction. Hence $\bfw=\mathbf0$ and $\bfv\in \operatorname{span}\{\bfu,\ci\bfu\}$, proving that the geometric multiplicity is $2$.\\[0.4em]
It remains to show that no Jordan chains occur at $\lambda$.  Suppose there were $\bfz\neq\mathbf0$ with
$(\mathcal{J}(\bfu)-\lambda\mathcal{I})\bfz=\mathcal{I} \boldsymbol{\gamma}$, where
$\boldsymbol{\gamma}\in\operatorname{span}\{\bfu,\ci\bfu\}\setminus\{\mathbf 0\}$.
Decompose $\bfz=a\,\bfu+b\,\ci\bfu+\bfw$ with
$\bfw\in T_{\bfu}\mathbb{S}\cap T_{\ci\bfu}\mathbb{S}$. Then
$(\mathcal{J}(\bfu)-\lambda\mathcal{I})\bfw=\mathcal{I} \boldsymbol{\gamma}$.
Taking the duality pairing with $\bfw$ and again using \eqref{identity-Ju-Eprimeprime}, we have
\begin{align*}
\big\langle (E''(\bfu)-\lambda\mathcal{I})\bfw,\bfw\big\rangle
\;=\;\big\langle (\mathcal{J}(\bfu)-\lambda\mathcal{I})\bfw,\bfw\big\rangle
\;=\;( \boldsymbol{\gamma},\bfw)_{\bfL^2(\D)} \;=\; 0,
\end{align*}
since $\bfw\perp_{\bfL^2}\boldsymbol{\gamma}$. By \eqref{A4-new} this forces $\bfw=\mathbf0$, hence
$\boldsymbol{\gamma}=\mathbf0$, a contradiction. Therefore, no Jordan chains exist.
\end{proof}

The lemma shows that the eigenspace of $\mathcal{J}(\bfu)$ associated with $\lambda$
is exactly two--dimensional, spanned by $\bfu$ and $\ci\bfu$.
These directions correspond to the natural symmetries of the problem:
real scaling and phase rotation. On the $\bfL^2$--orthogonal complement of these symmetry directions, the constrained Hessian is strictly positive due to assumption~\ref{A4}. Hence the degeneracy of $\lambda$ is entirely explained by the invariances.

The second property of the lemma ensures that $\lambda$ is semisimple, i.e., no Jordan chains occur for this eigenvalue. Consequently, the generalized eigenspace associated with $\lambda$ is precisely $\operatorname{span}_{\R}\{\bfu,\ci\bfu\}$. This structure is important for the subsequent spectral analysis of the $J$-method.

We now combine the G\r{a}rding inequality from assumption~\ref{A5} and the compact embedding
$\bfH_0^1(\D)\hookrightarrow \bfL^2(\D)$ to establish a Fredholm alternative
for the generally non-symmetric operator $\calJ{\bfu}$. For the spin--orbit coupled Gross--Pitaevskii model, assumption~\ref{A5} will be verified in the application section. Recall the following classical result (cf. \cite[Thrm. 2.2]{GH94} and \cite[Thrm. 2.34]{McLean}).
\begin{lemma}[Fredholm alternative]
\label{lem:Fredholm-alt-J}
Let $\mathcal{G} :\bfH_0^1(\D)\to \bfH^{-1}(\D)$ be the bounded linear operator
induced by the (in general non-symmetric) real bilinear form
\begin{align*}
  \langle \mathcal{G}\bfv,\bfw\rangle:\ \bfH_0^1(\D)\times \bfH_0^1(\D)\to\mathbb{R}.
\end{align*}
Assume there exist constants $\alpha,\beta>0$ such that the G\r{a}rding inequality
\begin{align*}
  \langle \mathcal{G}\bfv,\bfv\rangle \;\ge\;
  \alpha\,\|\bfv\|_{\bfH^1(\D)}^2 - \beta\,\|\bfv\|_{\bfL^2(\D)}^2
  \qquad\mbox{for all }\,\bfv\in\bfH_0^1(\D)
\end{align*}
holds. Then $\mathcal{G}$ is a Fredholm operator of index \(0\). Moreover, the Fredholm alternative holds, i.e., for every
$\sigma\in\mathbb{R}$ exactly one of the following alternatives holds:
\begin{enumerate}
  \item[(i)] There exists a nontrivial $\bfv\in\bfH_0^1(\D)$ such that
        \begin{align*}
          \mathcal{G}\bfv \;=\; \sigma\,\mathcal{I}\,\bfv \quad \text{in }\bfH^{-1}(\D).
        \end{align*}
  \item[(ii)] The shifted operator
        \begin{align*}
          \mathcal{G}-\sigma\,\mathcal{I}:\ \bfH_0^1(\D)\longrightarrow \bfH^{-1}(\D)
        \end{align*}
        is bijective and has a bounded inverse.
\end{enumerate}
\end{lemma}
From Lemma~\ref{lem:Fredholm-alt-J}, in combination with the G\r{a}rding assumption~\ref{A5}, the following statement can be made on the spectral structure of the operator $\calJ{\bfu}$. However, for this we need to distinguish between the original complex structure of our function spaces and their abstract complexification as real Hilbert spaces, see, e.g., \cite{Moslehian2022,Huhtanen2011}. For a real Hilbert space $X$, we denote its complexification by
$$
X_{\C}:=X\oplus \mathrm{i}_{\mathrm{ext}}X,
$$
where $\mathrm{i}_{\mathrm{ext}}$ is a formal imaginary unit, distinct from multiplication by $\ci$ in the original complex-valued function space. Every bounded real-linear operator $B:X\to Y$
has a canonical complex-linear extension $B_{\C}:X_{\C}\to Y_{\C}$ defined by
\begin{align*}
B_{\C}(p+\mathrm{i}_{\mathrm{ext}}q)
\,:=\, B\,p\,+\,\mathrm{i}_{\mathrm{ext}}B\,q.
\end{align*}
Note that the real eigenvalues of the original and complexified eigenvalue problems coincide. Indeed, every eigenfunction of the original problem is also an eigenfunction of the complexified problem. Conversely, for a real eigenvalue of the complexified problem, the real and imaginary parts of a corresponding eigenfunction both solve the original eigenvalue problem, with at least one of them being nonzero. Consequently, the complexification does not introduce additional real eigenvalues, but allows us to account for potentially nonreal spectral values and to apply the standard spectral theory for complex-linear operators.

Throughout the following analysis, complex eigenvalues and the complex spectrum of a real-linear operator are understood with respect to this abstract complexification. In particular, a complex eigenvalue
$\mu=a+\mathrm{i}_{\mathrm{ext}}b$, with $a,b\in\R$, refers to the eigenvalue problem for the complexified operator, where scalar multiplication is defined using $\mathrm{i}_{\mathrm{ext}}$. We use this construction for $\bfH_0^1(\D)$, $\bfL^2(\D)$, their dual spaces, and the operators
$\calJ{\bfu}$ and $\mathcal{I}$.
\begin{lemma}[Spectral structure of $\calJ{\bfu}$]\label{lem:spectral-structure-J}
Assume \ref{A1}--\ref{A3}, and let $(\bfu,\lambda)\in\mathbb{S}\times\R$ be a ground-state eigenpair satisfying~\ref{A5}. Then $\calJ{\bfu}:\bfH_0^1(\D)\to\bfH^{-1}(\D)$ is a Fredholm operator of index~$0$. In particular:
\begin{itemize}
  \item The complexified generalized eigenvalue problem seeking $\bfv\in(\bfH_0^1(\D))_{\C}\setminus\{\mathbf0\}$ with 
\begin{align*}
(\calJ{\bfu})_{\C}\bfv
=
\mu\,\mathcal{I}_{\C}\bfv
\quad\text{in }(\bfH^{-1}(\D))_{\C}
\end{align*}
has a discrete spectrum in $\C$, whose eigenvalues have finite algebraic multiplicity and no finite
accumulation points. 
\item The set of real eigenvalues of the original
real-linear problem is denoted by
\begin{align}
\label{def-ev-set-J-operator}
\Lambda(\bfu)
:=
\{\mu\in\R:
\exists\,\bfv\in\bfH_0^1(\D)\setminus\{\mathbf0\}
\text{ with }
\calJ{\bfu}\bfv=\mu\,\mathcal{I}\bfv\}.
\end{align}
For every $\sigma\in\R\setminus\Lambda(\bfu)$,
the shifted operator
$\calJ{\bfu}-\sigma\mathcal{I}:
\bfH_0^1(\D)\to\bfH^{-1}(\D)$
is bijective with a bounded inverse.
  \item Whenever $\calJsigma{\bfu} = \calJ{\bfu}-\sigma\mathcal{I}$ is invertible,
        the map
        \begin{align*}
          \calJsigma{\bfu}^{-1}\mathcal{I}:\bfL^2(\D)\to\bfH_0^1(\D)
        \end{align*}
        is compact.
\end{itemize}
\end{lemma}

\begin{proof}
By assumption~\ref{A5}, there exist $\alpha>0$ and $\beta\geq0$ such that
$$
\langle \calJ{\bfu}\bfv,\bfv\rangle \geq \alpha\|\bfv\|_{\bfH^1_0(\D)}^2-\beta\|\bfv\|_{\bfL^2(\D)}^2
\qquad\mbox{for all }\bfv\in\bfH_0^1(\D).
$$
Hence Lemma~\ref{lem:Fredholm-alt-J} applies to $\mathcal{G}:=\calJ{\bfu}$ and yields that
$\calJ{\bfu}:\bfH_0^1(\D)\to\bfH^{-1}(\D)$ is Fredholm of index $0$ and that, for every $\sigma\in\mathbb{R}$,
either $\sigma\in\Lambda(\bfu)$ or $\calJ{\bfu}-\sigma\mathcal{I}$ is bijective with bounded inverse.
In particular, for every $\sigma\in\mathbb{R}\setminus\Lambda(\bfu)$ the operator
$\calJsigma{\bfu}:=\calJ{\bfu}-\sigma\mathcal{I}$ is invertible.

Moreover, the embedding $\bfL^2(\D)\hookrightarrow\bfH^{-1}(\D)$ is compact as the adjoint of the compact embedding
$\bfH_0^1(\D)\hookrightarrow\bfL^2(\D)$. Therefore, whenever $\calJsigma{\bfu}$ is invertible, the map
$$
\calJsigma{\bfu}^{-1}\mathcal{I}:\bfL^2(\D)\to\bfH_0^1(\D)
$$
is compact, and by restriction also $\calJsigma{\bfu}^{-1}\mathcal{I}:\bfH_0^1(\D)\to\bfH_0^1(\D)$ is compact.

Now fix $\sigma\in\R\setminus\Lambda(\bfu)$ and consider the compact operator
\begin{align*}
K_\sigma := \calJsigma{\bfu}^{-1}\mathcal{I} : \bfH_0^1(\D)\to\bfH_0^1(\D).
\end{align*}
Its complexification $(K_\sigma)_{\C}$ is also compact. For $\mu\in\C\setminus\{\sigma\}$ and
$\bfv\in(\bfH_0^1(\D))_{\C}\setminus\{\mathbf0\}$, we have
\begin{align*}
(\calJ{\bfu})_{\C}\bfv = \mu\,\mathcal{I}_{\C}\bfv \quad\Longleftrightarrow\quad
(K_\sigma)_{\C}\bfv = (\mu-\sigma)^{-1}\bfv.
\end{align*}
Moreover, since $\mathcal{I}_{\C}: (\bfH_0^1(\D))_{\C}\to(\bfH^{-1}(\D))_{\C}$ is compact, $(\calJ{\bfu})_{\C}-\mu\mathcal{I}_{\C}$ is a compact perturbation of the bounded invertible operator $(\calJ{\bfu})_{\C}-\sigma\mathcal{I}_{\C}$.
Hence it is Fredholm of index zero and the failing of invertibility is equivalent to the existence of a nontrivial kernel.

Consequently, the eigenvalues of the complexified generalized eigenvalue problem correspond exactly
to the nonzero eigenvalues $\kappa$ of $(K_\sigma)_{\C}$ through the relation $\kappa=(\mu-\sigma)^{-1}$. This standard shift-invert correspondence also preserves algebraic multiplicities. Furthermore, since $(K_\sigma)_{\C}$ is compact, the Riesz--Schauder theory implies that all its nonzero spectral values are eigenvalues of finite algebraic multiplicity, with zero as their only possible accumulation point. Hence, the complexified generalized eigenvalue problem has a discrete spectrum
consisting of eigenvalues of finite algebraic multiplicity, without finite accumulation points.
\end{proof}

\begin{lemma}[Real spectrum of $\calJ{\bfu}$]
\label{lem:J-real-spectrum}
Assume \ref{A1}--\ref{A3} and let $(\bfu,\lambda)\in\mathbb{S}\times\R$ be a ground-state
eigenpair satisfying~\ref{A5}. Then every eigenvalue of the complexified generalized eigenvalue problem associated with $\calJ{\bfu}$ is real. Consequently, its full generalized spectrum
coincides with $\Lambda(\bfu)$.
\end{lemma}

\begin{proof}
Let $\mu=\Re(\mu) +\mathrm{i}_{\mathrm{ext}} \,\Im(\mu)$, with $\Re(\mu),\,\Im(\mu)\in\R$, be an eigenvalue of the complexified generalized eigenvalue problem. Let
\begin{align*}
\bfz=\boldsymbol{p} +\mathrm{i}_{\mathrm{ext}}\boldsymbol{q}
\in(\bfH_0^1(\D))_{\C}\setminus\{\mathbf0\}
\end{align*}
be a corresponding eigenfunction, where
$\boldsymbol{p},\boldsymbol{q}\in\bfH_0^1(\D)$. By definition of the complexification,
the equation
\begin{align*}
(\calJ{\bfu})_{\C}\bfz = \mu\,\mathcal{I}_{\C}\bfz
\end{align*}
is equivalent to the following two equations in the original real Hilbert-space setting:
\begin{align}
\calJ{\bfu}\boldsymbol{p}
&=
\Re(\mu)\,\mathcal{I}\boldsymbol{p} - \Im(\mu) \,\mathcal{I}\boldsymbol{q},
\label{eq:J-complex-real-1}\\
\calJ{\bfu}\boldsymbol{q}
&= \Im(\mu) \,\mathcal{I}\boldsymbol{p} +\Re(\mu) \,\mathcal{I}\boldsymbol{q}.
\label{eq:J-complex-real-2}
\end{align}
Because $\bfu\in\mathbb{S}$ is normalized, we can decompose
\begin{align*}
\boldsymbol{p}\,=\,\alpha_{\boldsymbol{p}} \,\bfu+\boldsymbol{p}_T,
\qquad
\boldsymbol{q}\,=\,\alpha_{\boldsymbol{q}}\,\bfu+\boldsymbol{q}_T,
\end{align*}
where $\alpha_{\boldsymbol{p}},\alpha_{\boldsymbol{q}}\in\R$ and 
$\boldsymbol{p}_T,\boldsymbol{q}_T \in T_{\bfu}\mathbb{S}$. Recall that
\begin{align*}
\calJ{\bfu}\bfu=\lambda\,\mathcal{I}\bfu.
\end{align*}
Testing \eqref{eq:J-complex-real-1} and \eqref{eq:J-complex-real-2} with some arbitrary $\bfv\in T_{\bfu}\mathbb{S}$, the radial contributions vanish because of $(\bfu,\bfv)_{\bfL^2(\D)}=0$. We therefore obtain
\begin{align*}
\langle\calJ{\bfu}\boldsymbol{p}_T,\bfv\rangle
&=
\Re(\mu)\,(\boldsymbol{p}_T,\bfv)_{\bfL^2(\D)}
-\Im(\mu)\,(\boldsymbol{q}_T,\bfv)_{\bfL^2(\D)},\\
\langle\calJ{\bfu}\boldsymbol{q}_T,\bfv\rangle
&=
\Im(\mu)\,(\boldsymbol{p}_T,\bfv)_{\bfL^2(\D)}
+\Re(\mu)\,(\boldsymbol{q}_T,\bfv)_{\bfL^2(\D)}.
\end{align*}
By \eqref{identity-Ju-Eprimeprime}, the action of $\calJ{\bfu}$ on tangent directions is identical to the action of $E''(\bfu)$. Hence
\begin{align*}
\langle E''(\bfu)\boldsymbol{p}_T,\bfv\rangle
&=
\Re(\mu)\,(\boldsymbol{p}_T,\bfv)_{\bfL^2(\D)}
-\Im(\mu)\,(\boldsymbol{q}_T,\bfv)_{\bfL^2(\D)},\\
\langle E''(\bfu)\boldsymbol{q}_T,\bfv\rangle
&=
\Im(\mu)\,(\boldsymbol{p}_T,\bfv)_{\bfL^2(\D)}
+\Re(\mu)\,(\boldsymbol{q}_T,\bfv)_{\bfL^2(\D)}.
\end{align*}
Choose $\bfv=\boldsymbol{q}_T$ in the first equation and $\bfv=\boldsymbol{p}_T$ in the second to get
\begin{align*}
\langle E''(\bfu)\boldsymbol{p}_T,
\boldsymbol{q}_T\rangle
&=
\Re(\mu)\,(\boldsymbol{p}_T,\boldsymbol{q}_T)_{\bfL^2(\D)}
- \Im(\mu) \, \|\boldsymbol{q}_T\|_{\bfL^2(\D)}^2,\\
\langle E''(\bfu)\boldsymbol{q}_T,
\boldsymbol{p}_T\rangle
&=
\Im(\mu) \, \|\boldsymbol{p}_T\|_{\bfL^2(\D)}^2
+\Re(\mu)\,(\boldsymbol{q}_T,\boldsymbol{p}_T)_{\bfL^2(\D)}.
\end{align*}
Since both $E''(\bfu)$ and $(\cdot,\cdot)_{\bfL^2(\D)}$ are symmetric, subtracting the two equations gives
\begin{align*}
\Im(\mu) \, \left( \|\boldsymbol{p}_T\|_{\bfL^2(\D)}^2 + \|\boldsymbol{q}_T\|_{\bfL^2(\D)}^2 \right)=0.
\end{align*}
If $\Im(\mu) \neq0$, then we necessarily have $\boldsymbol{p}_T=\boldsymbol{q}_T=\mathbf0$. 
Consequently,
$\boldsymbol{p}=\alpha_{\boldsymbol{p}} \,\bfu$ and
$\boldsymbol{q}=\alpha_{\boldsymbol{q}} \,\bfu$. Inserting these expressions into \eqref{eq:J-complex-real-1} and \eqref{eq:J-complex-real-2} and using $\calJ{\bfu}\bfu=\lambda\mathcal{I}\bfu$,
we obtain
\begin{align*}
\begin{pmatrix}
\lambda-\Re(\mu) & \Im(\mu) \\
-\Im(\mu) & \lambda-\Re(\mu)
\end{pmatrix}
\begin{pmatrix}
\alpha_{\boldsymbol{p}} \\ \alpha_{\boldsymbol{q}} 
\end{pmatrix}
=
\begin{pmatrix}
0\\0
\end{pmatrix}.
\end{align*}
However,
\begin{align*}
\det
\begin{pmatrix}
\lambda-\Re(\mu) & \Im(\mu) \\
-\Im(\mu) & \lambda-\Re(\mu)
\end{pmatrix}
=
(\lambda-\Re(\mu))^2+\Im(\mu)^2>0,
\end{align*}
since $\Im(\mu)\neq0$.
Therefore $\alpha_{\boldsymbol{p}}=\alpha_{\boldsymbol{q}}=0$, which implies
$\bfz=\mathbf0$ and hence contradicts the choice of an eigenfunction. We conclude that $\Im(\mu)=0$ and hence $\mu\in\R$.

Furthermore, if $\mu\in\R$ is an eigenvalue of the complexified problem, then \eqref{eq:J-complex-real-1} and \eqref{eq:J-complex-real-2} show that at least one of $\boldsymbol{p}$ or $\boldsymbol{q}$ is a nonzero eigenvector of the original real-linear problem. Thus $\mu\in\Lambda(\bfu)$.

Finally, by Lemma~\ref{lem:spectral-structure-J}, the generalized spectrum is discrete and
consists entirely of eigenvalues. Hence the full generalized spectrum coincides with $\Lambda(\bfu)$.
\end{proof}
We conclude that $\calJsigma{\bfu}=\calJ{\bfu}-\sigma\mathcal{I}$ is inf-sup stable on $\bfH^1_0(\D)$ for every shift $\sigma\notin \Lambda(\bfu)$. 
\begin{lemma}[inf--sup stability]\label{lem:infsup}
Assume \ref{A1}--\ref{A3}, and let $(\bfu,\lambda)\in\mathbb{S}\times\R$ be a ground-state eigenpair satisfying~\ref{A5}. Let $\sigma\notin\Lambda(\bfu)$.
Then there exists a constant $c_{\mathrm{is}}(\sigma,\bfu)>0$ such that
\begin{align*}
\inf_{\bfv\in\bfH_0^1(\D)\setminus\{\mathbf0\}}
\sup_{\bfw\in\bfH_0^1(\D)\setminus\{\mathbf0\}}
  \frac{\langle \calJsigma{\bfu}\bfv,\bfw\rangle}
       {\|\bfv\|_{\bfH^1(\D)}\,\|\bfw\|_{\bfH^1(\D)}}
  \;\ge\;
  c_{\mathrm{is}}(\sigma,\bfu).
\end{align*}
Consequently, for all $F\in\bfH^{-1}(\D)$,
\begin{align*}
  \|\calJsigma{\bfu}^{-1}F\|_{\bfH^1(\D)}
  \;\le\;
  c_{\mathrm{is}}(\sigma,\bfu)^{-1}\,\|F\|_{\bfH^{-1}(\D)}.
\end{align*}
In particular, for all $\bfv\in\bfL^2(\D)$,
\begin{align*}
  \|\calJsigma{\bfu}^{-1}\mathcal{I}\,\bfv\|_{\bfH^1(\D)}
  \;\lesssim\;
  c_{\mathrm{is}}(\sigma,\bfu)^{-1}\,\|\bfv\|_{\bfL^2(\D)}.
\end{align*}
\end{lemma}

\begin{proof}
The proof follows a standard Babu{\v{s}}ka--Ne{\v{c}}as inf--sup argument. By assumption~\ref{A5}, there exist $\alpha>0$ and $\beta\geq0$ such that, for all $\bfv\in\bfH_0^1(\D)$,
\begin{align*}
  \langle \calJ{\bfu}\bfv,\bfv\rangle
  \;\ge\;
  \alpha\|\bfv\|_{\bfH^1_0(\D)}^2 - \beta\|\bfv\|_{\bfL^2(\D)}^2
  \,=\,\alpha\|\bfv\|_{\bfH^1(\D)}^2-(\alpha+\beta)\|\bfv\|_{\bfL^2(\D)}^2.
\end{align*}
Fix $\nu:=\alpha+\beta+\max\{\sigma,0\}$. Then, for all $\bfv\in\bfH_0^1(\D)$,
\begin{align}\label{star-inf-sup}
  \langle (\calJsigma{\bfu}+\nu\mathcal{I})\bfv,\bfv\rangle
  = \langle \calJ{\bfu}\bfv,\bfv\rangle - \sigma \|\bfv\|_{\bfL^2(\D)}^2 + \nu\|\bfv\|_{\bfL^2(\D)}^2
  \;\ge\; \alpha\|\bfv\|_{\bfH^1(\D)}^2.
\end{align}
For $\sigma\notin\Lambda(\bfu)$,\, $\calJsigma{\bfu}:\bfH_0^1(\D)\to\bfH^{-1}(\D)$ is bijective; hence its adjoint
$\calJsigma{\bfu}^*:\bfH_0^1(\D)\to\bfH^{-1}(\D)$ is also bijective with bounded inverse
$(\calJsigma{\bfu}^*)^{-1}\in\mathcal{L}(\bfH^{-1}(\D),\bfH_0^1(\D))$. Given $\bfv\in\bfH_0^1(\D)$, define
\begin{align*}
  \bfz \;:=\; \nu\,(\calJsigma{\bfu}^*)^{-1}\,\mathcal{I}\bfv \;\in\; \bfH_0^1(\D),
  \qquad\text{so that}\qquad \calJsigma{\bfu}^* \bfz \;=\; \nu\,\mathcal{I}\bfv.
\end{align*}
With the test function $\bfw:=\bfv+\bfz$ we compute, using the definition of the adjoint and of $\mathcal{I}$,
\begin{align*}
  \langle \calJsigma{\bfu}\bfv,\bfw\rangle
  &= \langle \calJsigma{\bfu}\bfv,\bfv\rangle + \langle \calJsigma{\bfu}\bfv,\bfz\rangle
   = \langle \calJsigma{\bfu}\bfv,\bfv\rangle + \langle \bfv,\calJsigma{\bfu}^*\bfz\rangle \\
  &= \langle \calJsigma{\bfu}\bfv,\bfv\rangle + \nu\,\langle \bfv,\mathcal{I}\bfv\rangle
   = \langle (\calJsigma{\bfu}+\nu\mathcal{I})\bfv,\bfv\rangle
  \;\overset{\eqref{star-inf-sup}}{\ge}\; \alpha\|\bfv\|_{\bfH^1(\D)}^2.
\end{align*}
Poincar\'e's inequality and the boundedness of
$(\calJsigma{\bfu}^*)^{-1}$ imply
\begin{align*}
  \|\bfz\|_{\bfH^1(\D)}
  &\le \nu\,\|(\calJsigma{\bfu}^*)^{-1}\|_{\mathcal{L}(\bfH^{-1}(\D),\bfH_0^1(\D))}\,\|\mathcal{I}\bfv\|_{\bfH^{-1}(\D)} 
  \lesssim \nu\,\|(\calJsigma{\bfu}^*)^{-1}\|_{\mathcal{L}(\bfH^{-1}(\D),\bfH_0^1(\D))}\,\|\bfv\|_{\bfH^1(\D)}.
\end{align*}
Consequently, for some constant $\tilde{c}(\sigma,\bfu)
:= 1 + c\,\nu\,\|(\calJsigma{\bfu}^*)^{-1}\|_{\mathcal{L}(\bfH^{-1}(\D),\bfH_0^1(\D))}$, it holds
\begin{align*}
  \|\bfw\|_{\bfH^1(\D)}
  \le \|\bfv\|_{\bfH^1(\D)} + \|\bfz\|_{\bfH^1(\D)}
  \le \tilde{c}(\sigma,\bfu)\,\|\bfv\|_{\bfH^1(\D)}.
\end{align*}
Therefore,
\begin{align*}
\sup_{\bfw\in\bfH_0^1(\D)\setminus\{\mathbf0\}}
  \frac{\langle \calJsigma{\bfu}\bfv,\bfw\rangle}{\|\bfw\|_{\bfH^1(\D)}}
  \;\ge\;
  \frac{\langle \calJsigma{\bfu}\bfv,\bfv+\bfz\rangle}{\|\bfv+\bfz\|_{\bfH^1(\D)}}
  \;\ge\;
  \frac{\alpha\|\bfv\|_{\bfH^1(\D)}^2}{\tilde{c}(\sigma,\bfu)\,\|\bfv\|_{\bfH^1(\D)}}
  \;=\;
  \frac{\alpha}{\tilde{c}(\sigma,\bfu)}\,\|\bfv\|_{\bfH^1(\D)}.
\end{align*}
Dividing by $\|\bfv\|_{\bfH^1(\D)}$ and taking the infimum over $\bfv\neq \mathbf0$ yields the inf--sup stability with constant
$c_{\mathrm{is}}(\sigma,\bfu):=\alpha/\tilde{c}(\sigma,\bfu)>0$. The stability bound for $\calJsigma{\bfu}^{-1}$ follows from the Babu{\v{s}}ka--Ne{\v{c}}as theorem.
Finally, for $\bfv\in\bfL^2(\D)$ we have with Poincar\'e $\|\mathcal{I}\bfv\|_{\bfH^{-1}(\D)}\lesssim \|\bfv\|_{\bfL^2(\D)}$, hence
$\|\calJsigma{\bfu}^{-1}\mathcal{I}\,\bfv\|_{\bfH^1(\D)}
\lesssim  c_{\mathrm{is}}(\sigma,\bfu)^{-1}\,\|\bfv\|_{\bfL^2(\D)}$. 
\end{proof}
The next lemma establishes an identity in the gauge-direction which will be afterwards used to prove a particular phase invariance of the $J$-operator.
\begin{lemma}[Gauge-direction identity for $\calJ{\bfu}$]\label{lem:J-gauge}\quad\\
Assume \ref{A1}--\ref{A3} and let $\bfu\in\bfH^1_0(\D)$ be arbitrary with $\|\bfu\|_{\bfL^2(\D)}\neq0$.
Then it holds
\begin{align*}
\langle \calJ{\bfu}[\ci\bfu],\bfw\rangle
= \langle \calJ{\bfu}\bfu,-\ci\,\bfw\rangle
\qquad \mbox{for all } \bfw\in\bfH^1_0(\D).
\end{align*}
\end{lemma}
\begin{proof}
By the phase invariance~\eqref{phase-shifts-E} and the normalized energy identity~\eqref{normalized-energy-A}, we have
\begin{align*}
\langle \calA{\exp(\ci\theta)\bfu}[\exp(\ci\theta)\bfu],\bfw\rangle
\,=\, \langle \calA{\bfu}\bfu,\exp(-\ci\theta)\bfw\rangle
\qquad\text{for all }\theta\in\mathbb{R}.
\end{align*}
Differentiating with respect to $\theta$ at $\theta=0$ and using~\eqref{identity-A-J}, we obtain
\begin{align*}
\langle \calJ{\bfu}[\ci\bfu],\bfw\rangle
\,=\, \langle \calA{\bfu}\bfu,-\ci\,\bfw\rangle
\,=\, \langle \calJ{\bfu}\bfu,-\ci\,\bfw\rangle,
\end{align*}
which finishes the proof.
\end{proof}
Finally, the following lemma characterizes how the operator $\calJ{\bfu}\bfv$ reacts to phase changes.
\begin{lemma}[Phase identity for $\calJ{\bfu}$]\label{lem:J-phase-relation} \quad\\
Let assumptions~\ref{A1}--\ref{A3} hold and let $\bfu\in\bfH^1_0(\D) \setminus \{\mathbf{0}\}$.  Then, for all $\bfv,\bfw\in\bfH^1_0(\D)$ and $\omega\in\R$,
\begin{align}\label{eq:dual-phase-J}
  \big\langle \calJ{e^{\ci\omega}\bfu}[e^{\ci\omega}\bfv],\,\bfw\big\rangle
  \;=\;  \big\langle \calJ{\bfu}\bfv,\,e^{-\ci\omega}\bfw\big\rangle.
\end{align}
\end{lemma}

\begin{proof}
From the phase invariance~\eqref{phase-shifts-E} and the normalized energy identity~\eqref{normalized-energy-A}, we obtain for any $\bfz\in\bfH^1_0(\D)\setminus\{\mathbf{0}\}$ that
\begin{align}
\label{A-equivariance}
\langle \calA{e^{\ci\omega}\bfz}[e^{\ci\omega}\bfz],\bfw\rangle
= \langle \calA{\bfz}\bfz, e^{-\ci\omega}\bfw\rangle
\qquad\mbox{for all }\,\bfw\in\bfH^1_0(\D),~\omega\in\R.
\end{align}
Fix $\omega\in\R$ and $\bfu, \bfv,\bfw\in\bfH^1_0(\D)$ such that $\bfz=(\bfu+t\bfv) \not= \mathbf0$ in a neighbourhood of $t=0$, $t\in\R$. With this, consider the real-valued function
\begin{align*}
  \Phi(t) \;:=\;
  \big\langle \calA{e^{\ci\omega}(\bfu+t\bfv)}[e^{\ci\omega}(\bfu+t\bfv)],\,\bfw\big\rangle
  \;-\;  \big\langle \calA{\bfu+t\bfv}[\bfu+t\bfv],\,e^{-\ci\omega}\bfw\big\rangle,
  \quad \mbox{for } (\bfu+t\bfv) \not= \mathbf0.
\end{align*}
By the phase--equivariance \eqref{A-equivariance} we have $\Phi(t)=0$ for all $t\in\R$ in a neighbourhood of $t=0$. Differentiating $\Phi(t)=0$ and recalling that $\calJ{\bfz} = \frac{\mbox{d}}{\mbox{d}\bfz} [\calA{\bfz} \bfz]$ (cf. \eqref{def-J-op}), the chain rule yields for $t=0$ that
\begin{align*}
  0
  \;=\;\Phi'(0)
  &= \big\langle \calJ{e^{\ci\omega}\bfu}[e^{\ci\omega}\bfv],\,\bfw\big\rangle
   -  \big\langle \calJ{\bfu}\bfv,\,e^{-\ci\omega}\bfw\big\rangle,
\end{align*}
which finishes the proof.
\end{proof}

As a direct consequence of Lemma~\ref{lem:J-phase-relation}, we obtain:
\begin{lemma}[Phase identity for $\calJsigma{\bfu}^{-1}\mathcal{I}$]\label{lem:J-phase-relation-inverse}\quad\\
Assume~\ref{A1}--\ref{A3} and let $\bfu\in\bfH^1_0(\D)\setminus\{\mathbf{0}\}$. If $\calJsigma{\bfu}$ is invertible, then
\begin{align*}
  \calJsigma{e^{\ci\omega}\bfu}^{-1}\,\mathcal{I}\,[e^{\ci\omega}\bfu]
  \;=\;
  e^{\ci\omega}\,\calJsigma{\bfu}^{-1}\,\mathcal{I}\bfu
  \qquad \mbox{for all } \,\,\omega\in\R.
\end{align*}
\end{lemma}
\begin{proof}
Let $\bfy := \calJsigma{\bfu}^{-1}\,\mathcal{I}\bfu$, i.e., \,$\calJsigma{\bfu}\,\bfy \;=\; \mathcal{I}\bfu$\, in $\bfH^{-1}(\D)$. Equivalently, we have
\begin{align*}
  \langle \calJsigma{\bfu}\,\bfy,\,\bfw\rangle
  \;=\;
  \langle \mathcal{I}\bfu,\,\bfw\rangle
  \qquad \mbox{for all } \bfw\in\bfH^1_0(\D).
\end{align*}  
By Lemma~\ref{lem:J-phase-relation}, we have
\begin{align*}
  \langle \calJ{e^{\ci\omega}\bfu}[e^{\ci\omega}\bfy],\,e^{\ci\omega}\bfw\rangle
  \;=\;
  \langle \calJ{\bfu}\bfy,\,\bfw\rangle \qquad \mbox{for all } \bfw\in\bfH^1_0(\D).
\end{align*}
Together with the trivial identity $\big\langle \mathcal{I}\,[e^{\ci\omega}\bfy],\,e^{\ci\omega} \bfw\big\rangle\;=\;\big\langle \mathcal{I}\bfy,\,\bfw\big\rangle$ and $\calJsigma{\bfu} = \calJ{\bfu} - \sigma\,\mathcal{I}$ we obtain
\begin{align*}
  \langle \calJsigma{e^{\ci\omega}\bfu}[e^{\ci\omega}\bfy],\,e^{\ci\omega}\bfw\rangle
  \;=\;
  \langle \calJsigma{\bfu}\bfy,\,\bfw\rangle
  \;=\;
  \langle \mathcal{I}\bfu,\,\bfw\rangle
  \;=\;
  \langle \mathcal{I}\,[e^{\ci\omega}\bfu],\,e^{\ci\omega}\bfw\rangle.
\end{align*}
The map $\bfw\mapsto e^{\ci\omega}\bfw$ is a real-linear bijection of $\bfH^1_0(\D)$, hence the last identity implies in $\bfH^{-1}(\D)$ that
\begin{align*}
  \calJsigma{e^{\ci\omega}\bfu}\,[e^{\ci\omega}\bfy]
  \;=\;
  \mathcal{I}\,[e^{\ci\omega}\bfu].
\end{align*}
By Lemma~\ref{lem:J-phase-relation}, the operators $\calJsigma{\bfu}$ and $\calJsigma{e^{\ci\omega}\bfu}$ are related by phase conjugation. Hence, the invertibility of $\calJsigma{\bfu}$ also implies the invertibility of $\calJsigma{e^{\ci\omega}\bfu}$, and the above equation admits a unique solution in $\bfH^1_0(\D)$.  We conclude that
\begin{align*}
  \calJsigma{e^{\ci\omega}\bfu}^{-1}\,\mathcal{I}\,[e^{\ci\omega}\bfu]
  \;=\;
  e^{\ci\omega}\bfy
  \;=\;
  e^{\ci\omega}\,\calJsigma{\bfu}^{-1}\,\mathcal{I}\bfu,
\end{align*}
which proves the claim.
\end{proof}

\section{Proof of local convergence}
\label{sec:proof-loc-convergence}

Our goal is to prove locally linear convergence of the $J$-method for a fixed spectral shift $\sigma \in \R$. For this, we will exploit the well-known Ostrowski theorem (cf. \cite{AHP20,Shi81}) which we recall as follows in our setting:
\begin{lemma}[Ostrowski's theorem]\label{lemma:ostrowski}
Let $V\subset\bfH_0^1(\D)$ be an open set and let $\Phi:V\to\bfH_0^1(\D)$ define a map that is real Fr\'echet differentiable in a fixed point $\bfu\in V$, i.e., $\Phi(\bfu)=\bfu$. 
 Assume furthermore that the Fr\'echet derivative $\Phi'(\bfu):\bfH^1_0(\D)\to \bfH^1_0(\D)$ is a bounded real-linear operator with spectral radius
	\begin{align*}\rho := \rho(\Phi'(\bfu))<1.\end{align*}
	Then there exists an open neighbourhood $U$ of $\bfu$ such that the fixed-point iteration given by
	\begin{align*}\bfu^{n+1} := \Phi(\bfu^n)\end{align*}
	 converges strongly to $\bfu$ in $\bfH^1(\D)$ for all initial starting values $\bfu^0\in U$, i.e., $\|\bfu^n - \bfu\|_{\bfH^1(\D)} \to 0$ for $n\to \infty$. Additionally, there exists a neighbourhood $U_\eps$ of $\bfu$ for any $\eps>0$ such that
	 	\begin{align*}\|\bfu^n - \bfu\|_{\bfH^1(\D)} \leq C_\eps|\rho + \eps|^n\|\bfu^0 - \bfu\|_{\bfH^1(\D)}\, \quad \, \text{ for all }\, \bfu^0\in U_\eps \, \text{ and }\, n\geq 1.\end{align*}
	 	The spectral radius $\rho$ characterizes the sharp local linear convergence factor in the sense of uniform estimates for all sufficiently close initial values.
\end{lemma}
Recall that for a bounded real-linear operator $T$, its spectral radius is understood through its
complexification, i.e., $\rho(T):=\sup\{|\mu|:\mu\in\sigma(T_{\C})\}$.

To exploit Ostrowski's theorem we first express the $J$-method as a fixed-point iteration. However, as we will see, the canonical map has the property $\rho(\Phi_\sigma'(\bfu))\ge 1$ due to the invariance under global phase shifts. For that reason, the fixed-point maps need to be modified to make the theorem applicable.

We start with the canonical formulation in the next subsection. 

\subsection{Canonical fixed-point formulation}
Let us fix a ground-state eigenpair $(\bfu,\lambda)\in\mathbb{S}\times\mathbb{R}$ of the nonlinear eigenvalue problem~\eqref{eq:nonlinear-eigenvalue-problem} and a shift $\sigma\in\mathbb{R}\setminus\Lambda(\bfu)$. We consider the shifted operator $\calJsigma{\bfu}:=\calJ{\bfu}-\sigma\mathcal{I}:\bfH^1_0(\D)\to\bfH^{-1}(\D)$. Under assumption~\ref{A5}, Lemma~\ref{lem:spectral-structure-J} guarantees that $\calJsigma{\bfu}$ is invertible. Moreover, since the mapping $\bfv\mapsto\calJsigma{\bfv}$ is continuous from $\bfH_0^1(\D)\setminus\{\mathbf{0}\}$ into $\mathcal{L}(\bfH_0^1(\D),\bfH^{-1}(\D))$, the set
\begin{align*}
\mathcal U_\sigma
:=\big\{\bfv\in\bfH_0^1(\D)\setminus\{\mathbf{0}\}:
\calJsigma{\bfv}\mbox{ is invertible}\big\}
\end{align*}
is an open neighbourhood of $\bfu$ in $\bfH_0^1(\D)$. Furthermore, $\mathcal U_\sigma$ is invariant under global phase rotations by Lemma~\ref{lem:J-phase-relation}. We can therefore define the nonlinear map $\Psi_{\sigma}:\mathcal U_\sigma\to\bfH_0^1(\D)$ by
\begin{align}
\label{eq:psi}
\Psi_{\sigma}(\bfv) = (\lambda - \sigma) \calJsigma{\bfv}^{-1}\mathcal{I}\bfv
\end{align}
and obtain the corresponding fixed-point function $\Phi_{\sigma}:\mathcal{U}_\sigma \to \bfH_0^1(\D)$ as
\begin{equation}
	\label{eq:phi}
	\Phi_{\sigma}(\bfv) := \frac{\hspace{-24pt}\Psi_{\sigma}(\bfv)}{\|\Psi_{\sigma}(\bfv)\|_{\bfL^2(\D)}} 
	= \operatorname{sgn}(\lambda - \sigma)\frac{\hspace{-22pt}
		\calJsigma{\bfv}^{-1}\mathcal{I}\bfv}{\|
		\calJsigma{\bfv}^{-1} \mathcal{I} \bfv
		\|_{\bfL^2(\D)}}. 
\end{equation}
With this, the iterations of the $J$-method \eqref{J-iteration-shift} (for some initial value $\bfu^0\in\mathbb{S}$ and $n\geq0$) are given by
\begin{equation}
	\label{eq:iteration}
	\bfu^{n+1} = \Phi_{\sigma}(\bfu^n) = \operatorname{sgn}(\lambda - \sigma) \frac{\hspace{-24pt}
		\calJsigma{\bfu^n}^{-1}\mathcal{I}\bfu^n
	}{
	\|
	\calJsigma{\bfu^n}^{-1}\mathcal{I}\bfu^n
	\|_{\bfL^2(\D)}
	}.
\end{equation}
Note that for any eigenpair $(\bfu,\lambda)\in \mathbb{S} \times \R$ with $\calJ{\bfu} \bfu = \calA{\bfu} \bfu = \lambda \mathcal{I} \bfu$, we have
\begin{align*}
\Psi_{\sigma}(\bfu) &= 	(\lambda - \sigma)\calJsigma{\bfu}^{-1}\mathcal{I}\bfu=\calJsigma{\bfu}^{-1}\lambda\mathcal{I}\bfu - \sigma \calJsigma{\bfu}^{-1}\mathcal{I}\bfu
	\\
	&= \calJsigma{\bfu}^{-1}\calJ{\bfu}\bfu - \sigma \calJsigma{\bfu}^{-1}\mathcal{I}\bfu
	=\calJsigma{\bfu}^{-1}\calJsigma{\bfu}\bfu=\bfu.
\end{align*}
Hence, $\bfu$ is indeed a fixed point of $\Psi_{\sigma}$, as well as of $\Phi_{\sigma}$. 

We also recall that, practically speaking, the value of $\lambda$ is not known at the beginning of the iteration and hence, a statement on the sign of $\lambda - \sigma$ cannot be made in general. Mathematically speaking, the sign is important to ensure that a ground state $\bfu\in\mathbb{S}$ is a fixed point of $\Phi_{\sigma}$, as this is a necessary condition for making use of Ostrowski's theorem. However, computationally speaking, the sign does not pose a problem to finding ground states, as ignoring it will at most result in sign oscillations of $\bfu^{n}$ which will neither influence the convergence to the limiting density $|\bfu|^2$ nor to the target eigenvalue $\lambda$.

In order to understand the spectrum of $\Phi_{\sigma}'(\bfu)$, which is crucial for the local convergence of the fixed-point iteration, the next lemma provides expressions for $\Phi_{\sigma}'(\bfu)$ and $\Psi_{\sigma}'(\bfu)$, together with the critical observation that $\rho(\Phi_{\sigma}'(\bfu)) \not< 1$.
\begin{lemma}\label{lem:derivatives:psi-phi-prime}
Assume \ref{A1}--\ref{A3} and let $(\bfu,\lambda)\in\mathbb{S}\times\mathbb{R}$ denote an eigenpair of~\eqref{eq:nonlinear-eigenvalue-problem} such that $\calJsigma{\bfu}$ is invertible. Then, the real Fr\'echet derivatives of the maps $\Psi_{\sigma}$ and $\Phi_{\sigma}$ (see \eqref{eq:psi} and \eqref{eq:phi}) at $\bfu$ in direction $\bfv \in \bfH^1_0(\D)$ are given by 
\begin{eqnarray}
 \label{eq:psi_prime}      \Psi_{\sigma}'(\bfu)[\bfv] &=& (\lambda - \sigma)\calJsigma{\bfu}^{-1}\mathcal{I}\bfv, \\
 \label{eq:phi_prime}	     \Phi_{\sigma}'(\bfu)[\bfv] &=& (\lambda - \sigma)\big(\calJsigma{\bfu}^{-1}\mathcal{I}\bfv - (\calJsigma{\bfu}^{-1}\mathcal{I}\bfv,\bfu)_{\bfL^2(\D)}\bfu\big).
\end{eqnarray}
Furthermore, it holds $\Phi_{\sigma}'(\bfu) [\ci \bfu ]=\ci \bfu$, hence $\rho(\Phi_{\sigma}'(\bfu)) \ge 1$.
\end{lemma}
\begin{proof}
The proof essentially follows the same arguments as in \cite{AHP20}.

Since $\Psi_{\sigma}(\bfu)= \Phi_{\sigma}(\bfu) =\bfu$ and $ \| \bfu \|_{\bfL^2(\D)}= 1$, we can compute the derivative of $\Phi_{\sigma}$ in direction $\bfv\in\bfH_0^1(\D)$ with the product rule to obtain
\begin{eqnarray}
\label{eq:phi-prime-1}
\nonumber
	\Phi_{\sigma}'(\bfu)[\bfv] &=& \frac{\hspace{-17pt}\Psi_{\sigma}'(\bfu)[\bfv]}{\|\Psi_{\sigma}(\bfu)\|_{\bfL^2(\D)}}  -  \frac{\hspace{-25pt}\Psi_{\sigma}(\bfu)}{\|\Psi_{\sigma}(\bfu)\|_{\bfL^2(\D)}^2}(\Psi_{\sigma}'(\bfu)[\bfv], \Phi_{\sigma}(\bfu))_{\bfL^2(\D)} \\[0.5em]
	&=& \Psi_{\sigma}'(\bfu)[\bfv]  -  (\Psi_{\sigma}'(\bfu)[\bfv], \bfu )_{\bfL^2(\D)} \bfu.
\end{eqnarray}
To determine $\Psi_{\sigma}'(\bfu)[\bfv]$, recall that $\mathcal{I}\bfv = (\bfv, \cdot)_{\bfL^2(\D)}$ denotes the $\bfL^2(\D)$ identification. With this we have 
\begin{eqnarray*}
\mathcal{I}\bfv &=&\frac{\diff}{\diff \varepsilon}\mathcal{I}(\bfu+\varepsilon \bfv)\vert_{\varepsilon=0} 
\,\overset{\eqref{eq:psi}}{=}\, (\lambda - \sigma)^{-1}\frac{\diff}{\diff \varepsilon} \calJsigma{\bfu + \varepsilon \bfv}\Psi_{\sigma}(\bfu + \varepsilon \bfv)\vert_{\varepsilon=0} \\
&=& (\lambda - \sigma)^{-1}\left(\, ( \mathcal{J}_{\sigma}'(\bfu)[\bfv])\Psi_{\sigma}(\bfu)  \,+\, \calJsigma{\bfu}\,(\Psi_{\sigma}'(\bfu)[\bfv])\,\right).
\end{eqnarray*}
A simple rearrangement shows
\begin{eqnarray*}
\Psi_{\sigma}'(\bfu)[\bfv] \,\,=\,\, \calJsigma{\bfu}^{-1}\big((\lambda - \sigma)\mathcal{I}\bfv - ( \mathcal{J}_{\sigma}'(\bfu)[\bfv]) \Psi_{\sigma}(\bfu)\big)  \,\,=\,\, \calJsigma{\bfu}^{-1}\big((\lambda - \sigma)\mathcal{I}\bfv - ( \mathcal{J}_{\sigma}'(\bfu)[\bfv])\bfu\big).
\end{eqnarray*}
It remains to determine $(\mathcal{J}_{\sigma}'(\bfu)[\bfv])\bfu$. Differentiating the identity~\eqref{identity-A-J} with respect to $\bfu$ in direction $\bfv$ and recalling~\eqref{def-J-op}, we obtain
\begin{align*}
\mathcal{J}'(\bfu)[\bfv]\bfu + \calJ{\bfu}\bfv = \calJ{\bfu}\bfv,
\qquad\mbox{and consequently}\qquad
(\mathcal{J}_{\sigma}'(\bfu)[\bfv])\bfu = \mathcal{J}'(\bfu)[\bfv]\bfu = \mathbf{0}.
\end{align*}
Hence, the formula reduces to $\Psi_{\sigma}'(\bfu)[\bfv] = (\lambda - \sigma)\calJsigma{\bfu}^{-1}\mathcal{I}\bfv$, proving \eqref{eq:psi_prime}. Together with \eqref{eq:phi-prime-1} we also obtain \eqref{eq:phi_prime}.

It remains to show that $\bfv=\ci \bfu$ is an eigenfunction with eigenvalue $1$ to $\Phi_{\sigma}'(\bfu)$. Here we can directly apply Lemma \ref{lem:J-gauge} which yields for arbitrary  $\bfw\in\bfH^1_0(\D)$:
\begin{align*}
\langle \calJ{\bfu}[\ci\bfu],\bfw\rangle
\,\,&=\,\, \langle \calJ{\bfu}\bfu,-\ci\,\bfw\rangle 
\,\,=\,\,  \langle \calA{\bfu}\bfu,-\ci\,\bfw\rangle \\
\,\,&=\,\,  \lambda \, ( \bfu , -\ci\,\bfw)_{\bfL^2(\D)}
\,\,=\,\,  \lambda \, ( \ci  \bfu , \bfw)_{\bfL^2(\D)}
\,\,=\,\, \lambda \, \langle \mathcal{I}[\ci  \bfu] , \bfw \rangle.
\end{align*}
Consequently, we have from \eqref{eq:phi_prime}
\begin{eqnarray}
\nonumber \Phi_{\sigma}'(\bfu) [\ci \bfu]&=& (\lambda - \sigma)\big(\calJsigma{\bfu}^{-1}\mathcal{I}[\ci \bfu] - (\calJsigma{\bfu}^{-1}\mathcal{I}[\ci \bfu],\bfu)_{\bfL^2(\D)}\bfu\big) \\
 &=& (\lambda - \sigma)\big( \frac{1}{\lambda - \sigma} \ci \bfu - \frac{1}{\lambda - \sigma} (\ci \bfu,\bfu)_{\bfL^2(\D)}\bfu\big) \,\,\,=\,\,\,\ci \bfu. \label{phiprimasigma-ev-iu}
\end{eqnarray}
We conclude that $\ci \bfu$ is an eigenfunction of $ \Phi_{\sigma}'(\bfu) $ with eigenvalue $1$.
\end{proof}

As we just saw, it holds $\rho(\Phi_{\sigma}'(\bfu))\geq 1$, and Ostrowski's theorem is hence not directly applicable. This issue arises from the global phase invariance, which implies that the fixed point is not isolated. Indeed, since $\ci\bfu\in T_{\bfu}\mathbb{S}$, a perturbation in direction $\ci\bfu$ corresponds to an infinitesimal phase rotation and hence does not affect the energy. However, this issue can be resolved by fixing the phase, thereby eliminating the problematic direction $\ci\bfu$ from the iteration. We will now define an auxiliary iteration $\tilde{\Phi}_{\sigma}$ with a fixed phase that converges to the same ground state modulo global phase rotations.

\subsection{Auxiliary fixed-point formulation}
Before defining the aforementioned auxiliary iteration $\tilde{\Phi}_{\sigma}$, let us highlight which properties should be fulfilled. We still require a ground state $\bfu\in \S$ to be a fixed point, i.e., $\tilde{\Phi}_{\sigma}(\bfu)=\bfu$. Of course, the auxiliary iteration should be related to the original iteration given by \eqref{eq:iteration}. In our case, the auxiliary iteration given by $\tilde \bfu^{n+1}=\tilde{\Phi}_{\sigma}(\tilde \bfu^n)$ should produce an output that is merely a complex phase shift of the original iteration as long as both iterations are initialized with the same starting values. We would then like to establish a convergence rate for this auxiliary iteration, using Ostrowski's theorem. By the relationship between the two iterations, we will then be able to determine the convergence rate of \eqref{eq:iteration}. In order to use Ostrowski's theorem, we evidently require $\rho(\tilde{\Phi}_{\sigma}'(\bfu))<1$, implying that $\ci \bfu$ is no longer an eigenfunction with eigenvalue $1$ to $\tilde{\Phi}_{\sigma}'(\bfu)$. The following phase-lock strategy was first suggested in \cite{PHMY242} for gradient descent methods in a single-component setting.

\begin{definition}[Auxiliary iteration]
	For fixed $\bfu \in \S$ and an initial value $\tilde \bfu^0\in \S $,
	 the auxiliary iterates $\tilde \bfu^n\in \bfH^1_0(\D)$ (for $n\geq 0$) are given by
	\begin{equation}\label{eq:aux-iteration}
		\tilde \bfu^{n+1} =\tilde{\Phi}_{\sigma}(\tilde \bfu^n),
	\end{equation}
	where $\tilde{\Phi}_{\sigma}:\mathcal U_\sigma\to\bfH_0^1(\D)$ is defined by
	\begin{align*}
	\tilde{\Phi}_{\sigma}(\bfv) \,:=\, \Theta_{\sigma}(\bfv) \, {\Phi}_{\sigma}(\bfv)
	\end{align*}
         with the $\bfu$-dependent phase map $\Theta_{\sigma}:\mathcal{U}_\sigma\to S^1$ given as
	 \begin{align*}
	 \Theta_{\sigma}(\bfv):= \begin{cases}
	\frac{\int_{\D} \bfu \cdot \overline{\Psi_{\sigma}(\bfv)} d x}
	{\left| \int_{\D} \bfu \cdot \overline{\Psi_{\sigma}(\bfv)} dx\right|} \qquad &\mbox{for } \int_{\D} \bfu \cdot \overline{\Psi_{\sigma}(\bfv)} \dx \not= 0, \\[0.5em]
	\qquad 1 \qquad &\mbox{for } \int_{\D} \bfu \cdot \overline{\Psi_{\sigma}(\bfv)} \dx = 0.
	\end{cases}
	\end{align*}
\end{definition}
Multiplying $\Phi_{\sigma}(\bfv)$ by $\Theta_{\sigma}(\bfv)$ aligns the phase of $ \tilde{\Phi}_{\sigma}(\bfv)$ with the phase of $\bfu$. Let us observe a few properties of $\tilde{\Phi}_{\sigma}$: Recall that any eigenfunction $\bfu \in \mathbb{S}$ to the eigenvalue $\lambda$ is a fixed point of $\Psi_{\sigma}$. Since $\bfu\in \S$, it follows $\Theta_{\sigma}(\bfu)=1$ and thus $\tilde{\Phi}_{\sigma}(\bfu)=\bfu$. The key observation is that for arbitrary $\omega\in[-\pi,\pi)\setminus\{0\}$, a phase shift of $\bfu$ is still a fixed point of the original iteration, i.e., $\Phi_{\sigma}( \exp(\ci \omega)\,\bfu)=\exp(\ci \omega)\,\bfu$, but no longer a fixed point of the auxiliary iteration. In fact, we have $\Theta_{\sigma}(\exp(\ci \omega)\,\bfu)=\exp(-\ci \omega)\, $ and therefore 
\begin{align*}
\tilde{\Phi}_{\sigma}(\exp(\ci \omega)\,\bfu) 
= \Theta_{\sigma}(\exp(\ci \omega)\,\bfu) \, {\Phi}_{\sigma}(\exp(\ci \omega)\,\bfu)
= \exp(-\ci \omega)\,  \exp(\ci \omega)\, \bfu
= \bfu\neq \exp(\ci \omega)\, \bfu.
\end{align*} 

To clarify how a phase shift of the ground state affects the inverse of the shifted $J$-operator and to verify that the auxiliary iteration indeed just results in a phase shift of the original iteration, recall the statement of Lemma \ref{lem:J-phase-relation-inverse} which ensures for arbitrary $\bfv \in \mathcal{U}_\sigma$ that
\begin{align*}
  \calJsigma{\exp(\ci \omega)\,\bfv}^{-1}\,\mathcal{I}\,[\exp(\ci \omega)\,\bfv]
  \;=\;
  \exp(\ci \omega)\,\calJsigma{\bfv}^{-1}\,\mathcal{I}\bfv
  \qquad \mbox{for all } \,\,\omega\in\R.
\end{align*}
In particular, this provides us with the equation
\begin{align*}
\Psi_{\sigma}( \exp(\ci \omega)\,\bfv)\,=\,\exp(\ci \omega)\,\Psi_{\sigma}(\bfv)
\end{align*}
and consequently also
\begin{align}
\label{phase-identity-Phi-sigma}
\Phi_{\sigma}( \exp(\ci \omega)\,\bfv) = \frac{\hspace{-24pt}\Psi_{\sigma}( \exp(\ci \omega)\,\bfv)}{\| \Psi_{\sigma}( \exp(\ci \omega)\,\bfv) \|_{\bfL^2(\D)}}
= \frac{ \exp(\ci \omega)\, \Psi_{\sigma}(\bfv)}{\| \Psi_{\sigma}(\bfv)\|_{\bfL^2(\D)}}
= \exp(\ci \omega)\,\Phi_{\sigma}( \bfv).
\end{align}
With this in mind, we are now ready to relate the auxiliary iteration to the original iteration.

\begin{lemma}\label{lem:aux-it}
Assume \ref{A1}--\ref{A3}, let $(\bfu,\lambda)\in\mathbb{S}\times\R$ be an eigenpair of~\eqref{eq:nonlinear-eigenvalue-problem} and assume that $\calJsigma{\bfu}$ is invertible.
Consider the iterates of the $J$-method
\begin{align*}
\bfu^{n+1}=\Phi_{\sigma}(\bfu^n), \qquad n\ge 0,
\end{align*}
and the auxiliary iteration
\begin{align*}
\tilde \bfu^{\,n+1}
  =\tilde{\Phi}_{\sigma}(\tilde \bfu^{\,n})
  :=\Theta_{\sigma}(\tilde \bfu^{\,n})\Phi_{\sigma}(\tilde \bfu^{\,n}),
  \qquad n\ge 0,
\end{align*}
with the same initial value $\bfu^0=\tilde \bfu^{\,0}\in\S$.
If the iterations are well-defined for all $n\in \mathbb{N}$, then there exists a sequence $(\omega_n)_{n \in \N_0}\subset\R$ such that
\begin{align*}
  \tilde \bfu^{\,n} = \exp(\ci\omega_n)\,\bfu^n
  \qquad\text{for all }n\ge 0.
\end{align*}
In particular, $\tilde \bfu^{\,n}$ and $\bfu^n$ only differ by a complex phase factor for every $n\ge 0$ and hence
\begin{align*}
|\tilde \bfu^{\,n}(x)| = |\bfu^n(x)|  \quad\text{for a.e.\ }x\in\D.
\end{align*}
\end{lemma}

\begin{proof}
We use induction on $n$. Since $\tilde \bfu^{\,0}=\bfu^0$, the claim holds
for $n=0$ with $\omega_0:=0$.

Assume that for some $n\ge 0$ there exists $\omega_n\in\R$ such that
\begin{align}\label{eq:aux-it-ind-hyp}
  \tilde \bfu^{\,n} = \exp(\ci\omega_n)\,\bfu^n.
\end{align}
We show that the same holds for $n+1$. By definition of the auxiliary iteration,
\begin{align*}
\tilde \bfu^{\,n+1}
  =\Theta_{\sigma}(\tilde \bfu^{\,n})\,
    \Phi_{\sigma}(\tilde \bfu^{\,n}).
\end{align*}
Using the induction hypothesis \eqref{eq:aux-it-ind-hyp} and the
phase identity \eqref{phase-identity-Phi-sigma},
\begin{align*}
\Phi_{\sigma}(\tilde \bfu^{\,n})  =\Phi_{\sigma}(\exp(\ci\omega_n)\bfu^n)
 =\exp(\ci\omega_n)\,\Phi_{\sigma}(\bfu^n) =\exp(\ci\omega_n)\,\bfu^{n+1}.
 \end{align*}
Hence,
\begin{align*}
  \tilde \bfu^{\,n+1}
  &=\Theta_{\sigma}(\tilde \bfu^{\,n})\,
    \Phi_{\sigma}(\tilde \bfu^{\,n})
   =\Theta_{\sigma}(\tilde \bfu^{\,n})\,
     \exp(\ci\omega_n)\bfu^{n+1}.
\end{align*}
By definition, the phase map $\Theta_{\sigma}(\cdot)$ takes values in
$S^1=\{z\in\C:|z|=1\}$, so we can write
\begin{align*}
\Theta_{\sigma}(\tilde \bfu^{\,n}) = \exp(\ci\alpha_n)
\quad\text{for some }\alpha_n\in\R.
\end{align*}
Setting $\omega_{n+1}:=\omega_n+\alpha_n$, we obtain
\begin{align*}
\tilde \bfu^{\,n+1}
  =\exp(\ci(\omega_n+\alpha_n))\,\bfu^{n+1}
  =\exp(\ci\omega_{n+1})\,\bfu^{n+1},
\end{align*}
which is the desired relation for $n+1$. This completes the induction.\\[0.4em]
From $\tilde \bfu^{\,n} = \exp(\ci\omega_n)\bfu^n$ and $|\exp(\ci\omega_n)|=1$, we immediately have  $|\tilde \bfu^{\,n}| = |\bfu^n|$.
\end{proof}
This lemma will be used later in the final proof of convergence to connect the convergence rate from the auxiliary iteration to that of the original iteration. 

Next, we are interested in the spectral radius of $\tilde{\Phi}_{\sigma}'(\bfu)$ which characterizes the convergence rate of the auxiliary iteration. For this, we need to show that $\tilde{\Phi}_{\sigma}'(\bfu) : \bfH^1_0(\D) \rightarrow \bfH^1_0(\D)$ is a compact operator and hence every nonzero spectral value is an eigenvalue, with $0$ as the only possible accumulation point. Thus, we only need to consider nonzero eigenvalues when estimating the spectral radius. In the following, we will analyze $\tilde{\Phi}_{\sigma}'(\bfu)[\ci \bfu]$ and show that $\ci \bfu$ is an eigenfunction with eigenvalue $0$. This means that $\ci \bfu$ no longer poses a problem in our spectral radius. We will then show that there is no other eigenfunction with an eigenvalue of absolute value larger than $1$ provided that the shift $\sigma$ is sufficiently close to $\lambda$. 
Calculating the derivative $\tilde{\Phi}_{\sigma}'( \bfu) [\bfw]$ for arbitrary $ \bfw \in \bfH_0^1(\D)$ requires determining the derivatives of $\Phi_{\sigma}( \bfu)$ (which we established in Lemma \ref{lem:derivatives:psi-phi-prime}) and of $\Theta_{\sigma}( \bfu)$. It holds
\begin{align}
\label{eq:phi-aux-prime-v1}
\tilde{\Phi}_{\sigma}'( \bfu) [\bfw] = (\Theta_{\sigma}'( \bfu) [\bfw])\Phi_{\sigma}( \bfu) + \Theta_{\sigma}( \bfu)\Phi_{\sigma}'( \bfu) [\bfw] = (\Theta_{\sigma}'( \bfu) [\bfw]) \bfu + \Phi_{\sigma}'(\bfu) [\bfw]. 
\end{align}

\begin{lemma}\label{lem:aux-derivative-real}
Assume \ref{A1}--\ref{A3}. 
Let $(\bfu,\lambda)\in\S \times \R$ be an eigenpair of~\eqref{eq:nonlinear-eigenvalue-problem} and assume that $\calJsigma{\bfu}$ is invertible.
Then $\Theta_{\sigma}$ is Fr\'echet differentiable at $\bfu$ and, for every 
$\bfw\in\bfH_0^1(\D)$, it holds
\begin{align*}
  \Theta_{\sigma}'(\bfu)[\bfw]
  =  -\ci\,\bigl(\Psi_{\sigma}'(\bfu)[\bfw],\ci\bfu\bigr)_{\bfL^2(\D)}. 
\end{align*}
\end{lemma}

\begin{proof}
As $\calJsigma{\bfu}$ is invertible, $\Psi_{\sigma}(\bfv)$ is well-defined and Fr\'echet differentiable in a neighbourhood of $\bfu$. Define the complex-valued functional
\begin{align*}
  G(\bfv)
  := \int_{\D} \bfu \cdot \overline{\Psi_{\sigma}(\bfv)} \dx.
\end{align*}
By construction, we have $G(\bfu)=1$ and
\begin{align*}
  \Theta_{\sigma}(\bfv) = \frac{G(\bfv)}{|G(\bfv)|}
\end{align*}
for $\bfv$ in a neighbourhood of $\bfu$ where $G(\bfv)\neq 0$. Thus $\Theta_{\sigma}=g\circ G$
with $g:\C\setminus\{0\}\to S^1$ being $g(z):=z/|z|$. Let $z_0:=G(\bfu)=1$
and $h:=G'(\bfu)\bfw$. A standard computation of the
real derivative of $g$ at $z_0\in S^1$ gives
\begin{align*}
  g'(z_0) h = h - \Re(\overline{z_0} h) \, z_0  = h - \Re(h) = \ci\,\Im(h).
\end{align*}
Using the chain rule,
\begin{align*}
  \Theta_{\sigma}'(\bfu)[\bfw]
  = g'(G(\bfu))[G'(\bfu)[\bfw]]
  = \ci\,\Im\big(G'(\bfu)[\bfw]\big).
\end{align*}
Now write \,\, $\displaystyle z := \int_{\D} \Psi_{\sigma}'(\bfu)[\bfw] \cdot \overline{\bfu} \dx$,\,\, so that \,\,$G'(\bfu)[\bfw] = \overline{z}$.\,\,Then
\begin{align*}
  \Im\big(G'(\bfu)[\bfw]\big)
  = \Im(\overline{z})
  = - \Im(z)
  = \Re(\ci z) 
  = - \bigl(\Psi_{\sigma}'(\bfu)[\bfw],\ci\bfu\bigr)_{\bfL^2(\D)}. 
\end{align*}
Hence,
\begin{align}\label{theta:prime}
  \Theta_{\sigma}'(\bfu)[\bfw]
  = \ci\,\Im\big(G'(\bfu)[\bfw]\big)
  = -\ci\,\bigl(\Psi_{\sigma}'(\bfu)[\bfw],\ci\bfu\bigr)_{\bfL^2(\D)}.
\end{align}
\end{proof}

\noindent
Combining \eqref{eq:phi-aux-prime-v1} with
\eqref{theta:prime} yields the simplified expression
\begin{align}
  \tilde{\Phi}_{\sigma}'(\bfu)[\bfw]
  = - \ci\,\bigl(\Psi_{\sigma}'(\bfu)[\bfw],\ci\bfu\bigr)_{\bfL^2(\D)}\,\bfu
    \;+\; \Phi_{\sigma}'(\bfu)[\bfw], 
  \label{eq:phi-aux-prime}
\end{align}
which holds for eigenfunctions $\bfu\in\mathbb{S}$ of~\eqref{eq:nonlinear-eigenvalue-problem}. We now prove that 
$\tilde{\Phi}_{\sigma}'(\bfu):\bfH_0^1(\D)\to\bfH_0^1(\D)$
is a compact operator. 

\begin{lemma}[Compactness of $\tilde{\Phi}_{\sigma}'(\bfu)$]\label{lem:compact-tildePhi}\quad\\
Assume \ref{A1}--\ref{A3}, let $(\bfu,\lambda)\in\S\times\R$ be an eigenpair of~\eqref{eq:nonlinear-eigenvalue-problem} and assume that $\calJsigma{\bfu}$ is invertible. Then
$$
\tilde{\Phi}_{\sigma}'(\bfu):\bfH_0^1(\D)\to\bfH_0^1(\D)
$$
is compact. 
\end{lemma}

\begin{proof}
Using \eqref{eq:phi-aux-prime} we estimate, for arbitrary 
$\bfw\in\bfH_0^1(\D)$,
\begin{eqnarray*}
  \lefteqn{ \|\tilde{\Phi}_{\sigma}'(\bfu)[\bfw]\|_{\bfH^1(\D)}
  \,\,\,\le\,\,\, 
  \|\Phi_{\sigma}'(\bfu)[\bfw]\|_{\bfH^1(\D)}
  \;+\;
  \bigl|
    (\Psi_{\sigma}'(\bfu)[\bfw],\ci\bfu)_{\bfL^2(\D)}
  \bigr|
  \,\|\bfu\|_{\bfH^1(\D)} } \\[0.3em]
  &\lesssim&  \|\Phi_{\sigma}'(\bfu)[\bfw]\|_{\bfH^1(\D)}
  \;+\;
  \| \Psi_{\sigma}'(\bfu)[\bfw]\|_{\bfL^2(\D)} \\
  &\overset{\eqref{eq:psi_prime},\eqref{eq:phi_prime}}{=}&
    |\lambda - \sigma| \left( \| \calJsigma{\bfu}^{-1}\mathcal{I}\bfw - (\calJsigma{\bfu}^{-1}\mathcal{I}\bfw,\bfu)_{\bfL^2(\D)}\bfu \|_{\bfH^1(\D)}
  \;+\;
  \| \calJsigma{\bfu}^{-1}\mathcal{I}\bfw \|_{\bfL^2(\D)} \right) \\
  &\overset{\lambda \not= \sigma}{\lesssim}& \| \calJsigma{\bfu}^{-1}\mathcal{I}\bfw \|_{\bfH^1(\D)}.
\end{eqnarray*}
Since $\calJsigma{\bfu}$ is boundedly invertible and $\mathcal{I}:\bfL^2(\D)\to\bfH^{-1}(\D)$ is bounded, the operator $\calJsigma{\bfu}^{-1}\mathcal{I}:\bfL^2(\D)\to\bfH_0^1(\D)$ is bounded. Hence,
\begin{align*}
\|\tilde{\Phi}_{\sigma}'(\bfu)[\bfw]\|_{\bfH^1(\D)} \;\lesssim\;
\| \calJsigma{\bfu}^{-1}\mathcal{I}\bfw \|_{\bfH^1(\D)}
  \;\lesssim\;
  \|\bfw\|_{\bfL^2(\D)} .
\end{align*}
In particular, we have 
$\|\tilde{\Phi}_{\sigma}'(\bfu)[\bfw]\|_{\bfH^1(\D)}\lesssim \|\bfw\|_{\bfH^1(\D)}$. 
Now let $(\bfw_n)_{n\in\mathbb{N}}$ be a bounded sequence in $\bfH_0^1(\D)$.
By the Rellich--Kondrachov theorem, there exists a subsequence (for simplicity still denoted by $\bfw_n$) 
such that
\begin{align*}
  \bfw_n \rightharpoonup \bfw
  \quad\text{in } \bfH^1(\D),
  \qquad 
  \bfw_n \to \bfw
  \quad\text{in } \bfL^2(\D) .
\end{align*}
Since the preceding estimate shows that $\tilde{\Phi}_{\sigma}'(\bfu)$ extends continuously from $\bfL^2(\D)$ to 
$\bfH^1(\D)$ (i.e., $\|\tilde{\Phi}_{\sigma}'(\bfu)[\bfw]\|_{\bfH^1(\D)} \lesssim \|\bfw\|_{\bfL^2(\D)}$), we obtain
\begin{align*}
  \tilde{\Phi}_{\sigma}'(\bfu)[\bfw_n]
  \;\longrightarrow\;
  \tilde{\Phi}_{\sigma}'(\bfu)[\bfw]
  \quad\text{in } \bfH^1(\D).
\end{align*}
Hence the image of every bounded sequence under 
$\tilde{\Phi}_{\sigma}'(\bfu)$
contains a strongly convergent subsequence in $\bfH^1(\D)$.
This proves that $\tilde{\Phi}_{\sigma}'(\bfu) :\bfH_0^1(\D)\to\bfH_0^1(\D)$ is compact.
\end{proof}

We have shown that, for each eigenfunction $\bfu\in\S$ for which $\calJsigma{\bfu}$ is invertible, the operator $\tilde{\Phi}_\sigma'(\bfu):\bfH_0^1(\D)\to\bfH_0^1(\D)$ is compact. Therefore, every nonzero spectral value is an eigenvalue of $\tilde{\Phi}_\sigma'(\bfu)$, with $0$ as the only possible accumulation point. Hence, in the spectral analysis of $\tilde{\Phi}_\sigma'(\bfu)$ it suffices to consider nonzero eigenvalues.

Recall that for the local convergence analysis of the $J$-method we require the spectral radius to fulfill
\begin{align*}
\rho(\tilde{\Phi}_{\sigma}'(\bfu))<1 
\end{align*}
at the eigenfunction $\bfu$. For the original iteration \eqref{eq:iteration}, this condition could \emph{not} be satisfied (cf. \eqref{phiprimasigma-ev-iu}) because $\Phi_{\sigma}'(\bfu)[\ci\bfu]=\ci\bfu$, i.e.,\ $\ci\bfu$ was an eigenfunction with eigenvalue $1$. This eigenvalue arises from the global phase invariance of the energy functional $E$.

On the contrary, using the identity \eqref{eq:psi_prime}, i.e. $\Psi_{\sigma}'(\bfu)[\bfw]
  = (\lambda-\sigma)\,\calJsigma{\bfu}^{-1}\mathcal{I}\bfw$, we evaluate the real $\bfL^2$--inner product appearing in \eqref{eq:phi-aux-prime} for $\bfw=\ci\bfu$ to obtain
\begin{align}
\bigl(\Psi_{\sigma}'(\bfu)[\ci\bfu],\,\ci\bfu\bigr)_{\bfL^2(\D)}
&= (\lambda-\sigma)\bigl(\calJsigma{\bfu}^{-1}\mathcal{I}[\ci\bfu],\,\ci\bfu\bigr)_{\bfL^2(\D)} .
\label{eq:psi-ci-u}
\end{align}
We want to show that this term is $1$. For that, we use the phase identity in Lemma~\ref{lem:J-phase-relation-inverse} to see that
\begin{align*}
\Psi_{\sigma}( \exp(\ci\omega)\,\bfu)
=\exp(\ci\omega)\,\Psi_{\sigma}(\bfu)
=\exp(\ci\omega)\,\bfu.
\end{align*}
Differentiating with respect to $\omega$ at $\omega=0$ and using \eqref{eq:psi_prime}, we obtain
\begin{align*}
(\lambda-\sigma)\calJsigma{\bfu}^{-1}\mathcal{I}[\ci\bfu] =\Psi_{\sigma}'(\bfu)[\ci\bfu] =\ci\bfu.
\end{align*}
Hence, \eqref{eq:psi-ci-u} becomes
\begin{align*}
\bigl(\Psi_{\sigma}'(\bfu)[\ci\bfu],\,\ci\bfu\bigr)_{\bfL^2(\D)} = \frac{ \lambda-\sigma}{ \lambda-\sigma} \, \| \ci \bfu \|_{\bfL^2(\D)}^2
  = 1.
\end{align*}
Now we use \eqref{eq:phi-aux-prime}, i.e., $\tilde{\Phi}_{\sigma}'(\bfu)[\bfw] = - \ci\,\bigl(\Psi_{\sigma}'(\bfu)[\bfw],\ci\bfu\bigr)_{\bfL^2(\D)}\,\bfu + \Phi_{\sigma}'(\bfu)[\bfw] $, 
to obtain
\begin{align*}
\tilde{\Phi}_{\sigma}'(\bfu)[\ci \bfu]
  \,\,=\,\, - \ci\,\bigl(\Psi_{\sigma}'(\bfu)[\ci \bfu],\ci\bfu\bigr)_{\bfL^2(\D)}\,\bfu
    \;+\; \Phi_{\sigma}'(\bfu)[\ci \bfu ]
    \,\,=\,\, - \ci \bfu
    \;+\; \Phi_{\sigma}'(\bfu)[\ci \bfu ]\overset{\Phi_{\sigma}'(\bfu)[\ci\bfu]=\ci\bfu}{=}  \mathbf0.
\end{align*}
Hence,\,\, $\tilde{\Phi}_{\sigma}'(\bfu)[\ci\bfu]=\mathbf 0$,\,\, i.e.,\ $\ci\bfu$ is now an eigenfunction with eigenvalue $0$.  
Hence the problematic eigenvalue $1$ no longer appears in the derivative of the iteration map, which is essential for achieving $\rho(\tilde{\Phi}_{\sigma}'(\bfu))<1$ in Ostrowski's theorem.

Let us summarize the spectral consequences. By Lemma~\ref{lem:compact-tildePhi},
the operator $\tilde{\Phi}_{\sigma}'(\bfu)$ is compact on $\bfH_0^1(\D)$. Moreover, the operator
$\calJsigma{\bfu}^{-1}\mathcal{I}$ is compact by Lemma~\ref{lem:spectral-structure-J}.
Consequently, the complexifications of both operators are compact, and their nonzero spectral values are isolated eigenvalues of finite algebraic multiplicity, with zero as the only possible accumulation point. Their nonzero spectra also coincide with those of their respective adjoints up to complex conjugation.

Furthermore, Lemma~\ref{lem:J-real-spectrum} shows that, at a ground-state eigenfunction $\bfu$ satisfying assumption~\ref{A5},
all spectral values of the complexified generalized $J$-eigenvalue problem are real. By the shift-invert correspondence, the nonzero spectral values of the complexified operator $\calJsigma{\bfu}^{-1}\mathcal I$ are therefore real as well.

In what follows, we establish orthogonality properties of adjoint eigenfunctions that are useful for relating the $J$-operator to the constrained Hessian. The spectral correspondence with
$\tilde{\Phi}_{\sigma}'(\bfu)$ will subsequently be derived directly using the decomposition into
the symmetry directions and their $\bfL^2$--orthogonal complement.

\begin{lemma}\label{lem:J-adjoint}
Assume \ref{A1}--\ref{A3}, let $(\bfu,\lambda)\in\S\times\R$ be an eigenpair of~\eqref{eq:nonlinear-eigenvalue-problem} and assume that $\calJsigma{\bfu}$ is invertible. Let $\mu\neq\lambda$ be a real eigenvalue of $\calJ{\bfu}$, i.e., there exists
$\bfz \neq \mathbf 0$ such that
$$
  \calJ{\bfu}\bfz = \mu\,\mathcal{I}\bfz.
$$
Then $(\mu - \sigma)^{-1}$ is an eigenvalue of the compact operator
$$
  \calJsigma{\bfu}^{-1}\mathcal{I}:\bfL^2(\D)\to \bfL^2(\D)
$$
and of its Hilbert-space adjoint 
$\bigl(\calJsigma{\bfu}^{-1}\mathcal{I}\bigr)^\ast:\bfL^2(\D)\to \bfL^2(\D)$.
Let $\bfz^\ast\in\bfL^2(\D)\setminus\{\mathbf 0\}$ be a corresponding adjoint eigenfunction, i.e.,
$$
  \bigl(\calJsigma{\bfu}^{-1}\mathcal{I}\bfv,\bfz^\ast\bigr)_{\bfL^2(\D)}
  = (\mu - \sigma)^{-1}\,(\bfv,\bfz^\ast)_{\bfL^2(\D)}
  \qquad\mbox{for all }\,\bfv\in \bfL^2(\D).
$$
Then $\bfz^\ast\in \bfH_0^1(\D)$ and
$\bfz^\ast\in T_{\bfu}\S\cap T_{\ci\bfu}\S$.
\end{lemma}

\begin{proof}
Since $\calJsigma{\bfu}:\bfH_0^1(\D)\to\bfH^{-1}(\D)$ is boundedly invertible, the same holds for its variational adjoint $\calJsigma{\bfu}^{*}$. Moreover, the $\bfL^2$--adjoint of $\calJsigma{\bfu}^{-1}\mathcal I$ is given by
\begin{align*}
\bigl(\calJsigma{\bfu}^{-1}\mathcal I\bigr)^{*}
=
\bigl(\calJsigma{\bfu}^{*}\bigr)^{-1}\mathcal I,
\end{align*}
where the right-hand side maps $\bfL^2(\D)$ continuously into $\bfH_0^1(\D)$. Since $(\mu-\sigma)^{-1}\neq0$, the adjoint eigenrelation implies $\bfz^\ast\in\bfH_0^1(\D)$.

Since $\calJsigma{\bfu}^{-1}\mathcal{I}$ is compact and $(\mu-\sigma)^{-1}$ is a nonzero real eigenvalue, this eigenvalue also belongs to the spectrum of its Hilbert-space adjoint and admits a corresponding nonzero real eigenfunction $\bfz^\ast$.

Taking $\bfv=\bfu$ in the adjoint eigenrelation gives
$$
  \bigl(\calJsigma{\bfu}^{-1}\mathcal{I}\bfu,\bfz^\ast\bigr)_{\bfL^2(\D)}
  = (\mu - \sigma)^{-1}(\bfu,\bfz^\ast)_{\bfL^2(\D)}.
$$
Since $(\bfu,\lambda)$ is an eigenpair of $\calJ{\bfu}$, we have
$$
  \calJ{\bfu}\bfu = \lambda\,\mathcal{I}\bfu
  \quad\Longrightarrow\quad
  \calJsigma{\bfu}\bfu = (\lambda - \sigma)\,\mathcal{I}\bfu
  \quad\Longrightarrow\quad
  \calJsigma{\bfu}^{-1}\mathcal{I}\bfu = (\lambda - \sigma)^{-1}\bfu.
$$
Inserting this into the previous identity yields
$$
  (\lambda - \sigma)^{-1}(\bfu,\bfz^\ast)_{\bfL^2(\D)}
  = (\mu - \sigma)^{-1}(\bfu,\bfz^\ast)_{\bfL^2(\D)}.
$$
Since $\mu\neq\lambda$, it follows that
$(\bfu,\bfz^\ast)_{\bfL^2(\D)}=0$, i.e., $\bfz^\ast\in T_{\bfu}\S$.

By the gauge-direction identity for $\calJ{\bfu}$ (cf. Lemma \ref{lem:J-gauge}), $\ci\bfu$ is also an eigenfunction of
$\calJ{\bfu}$ with the same eigenvalue $\lambda$, hence
$$
  \calJ{\bfu}[\ci\bfu ]= \lambda\,\mathcal{I}[\ci\bfu]
  \quad\Longrightarrow\quad
  \calJsigma{\bfu}[\ci\bfu ]= (\lambda - \sigma)\,\mathcal{I}[\ci\bfu]
  \quad\Longrightarrow\quad \calJsigma{\bfu}^{-1}\mathcal{I}[\ci\bfu]
  = (\lambda - \sigma)^{-1}\ci\bfu.
$$
Choosing $\bfv=\ci\bfu$ in the adjoint eigenrelation gives
$$
  (\lambda - \sigma)^{-1}(\ci\bfu,\bfz^\ast)_{\bfL^2(\D)}
  = (\mu - \sigma)^{-1}(\ci\bfu,\bfz^\ast)_{\bfL^2(\D)}.
$$
Again using $\mu\neq\lambda$, we obtain
$(\ci\bfu,\bfz^\ast)_{\bfL^2(\D)}=0$, i.e.,
$\bfz^\ast\in T_{\ci\bfu}\S$. Thus $\bfz^\ast\in T_{\bfu}\S\cap T_{\ci\bfu}\S$.
\end{proof}

As a side note, the tangential identity \eqref{identity-Ju-Eprimeprime} and the second-order necessary condition for minimizers imply that $\lambda$ is the smallest real eigenvalue of $\calJ{\bfu}$ if $\bfu$ is a minimizer of $E$.

\begin{conclusion}\label{concl:lambda-smallest-eigenvalue}
Assume \ref{A1}--\ref{A3} and let $(\bfu,\lambda)\in\S\times\R$ be an eigenpair with $\bfu$ a (local) minimizer of $E$ on $\mathbb{S}$. Then $\lambda$ is the smallest real eigenvalue of $\calJ{\bfu}$ in the following sense:\\[0.3em]
It holds $\calJ{\bfu}\bfu =  \lambda\,\mathcal{I}\bfu$ and 
if $\mu\in\R$ is another eigenvalue of $\calJ{\bfu}$, i.e.,
there exists $\bfz\in \mathbb{S}$ such that
\begin{equation}
\label{eq:J-eig-real}
  \calJ{\bfu}\bfz = \mu\,\mathcal{I}\bfz,
\end{equation}
then $\mu\ge\lambda$.
\end{conclusion}

\begin{proof}
Let $\mu\in\R$ be a real eigenvalue of $\calJ{\bfu}$ with $\mu\neq\lambda$, i.e., WLOG there exists $\bfz\in\mathbb{S}$ fulfilling \eqref{eq:J-eig-real}. Since $\bfu\in\S$ is normalized, we can decompose
\begin{align*}
\bfz=\alpha\bfu+\bfw,
\qquad
\mbox{where }\alpha\in\R
\quad\mbox{and}\quad
\bfw\in T_{\bfu}\S.
\end{align*}
Note that $\bfw\neq\mathbf0$, since otherwise $\bfz$ would be a nonzero multiple of $\bfu$, which would imply $\mu=\lambda$. 
Testing \eqref{eq:J-eig-real} with $\bfw$ and using $\calJ{\bfu}\bfu=\lambda\mathcal I\bfu$, we obtain
\begin{eqnarray*}
\mu\,\|\bfw\|_{\bfL^2(\D)}^2
&=& \mu\,(\bfz,\bfw)_{\bfL^2(\D)}
\,\,=\,\, \langle\calJ{\bfu}\bfz,\bfw\rangle 
\,\,=\,\, \alpha\lambda\,(\bfu,\bfw)_{\bfL^2(\D)} +\langle\calJ{\bfu}\bfw,\bfw\rangle\\
&=& \langle\calJ{\bfu}\bfw,\bfw\rangle
\overset{\eqref{identity-Ju-Eprimeprime}}{=}
\langle E''(\bfu)\bfw,\bfw\rangle
\,\,\overset{\eqref{necessary-cond-min}}{\ge}\,\,
\lambda\,\|\bfw\|_{\bfL^2(\D)}^2.
\end{eqnarray*}
Since $\bfw\neq\mathbf0$, we conclude that $\mu\ge\lambda$, which proves the claim.
\end{proof}

Next, we state how the spectrum of $\tilde{\Phi}_{\sigma}'(\bfu)$ is related to $\bfu$ and $\ci \bfu$.
\begin{lemma}\label{lem:aux-phi-prime-adjoint}
Assume \ref{A1}--\ref{A3}, and let $(\bfu,\lambda)\in\S \times \R$ be an eigenpair of~\eqref{eq:nonlinear-eigenvalue-problem}. Let $\sigma\in\R$ be such that $\calJsigma{\bfu}$ is boundedly invertible.
Let $\mu\neq0$ be a (real) eigenvalue of the compact extension of
$\tilde{\Phi}_{\sigma}'(\bfu)$ to $\bfL^2(\D)$, and hence also of its
Hilbert-space adjoint $(\tilde{\Phi}_{\sigma}'(\bfu))^\ast:\bfL^2(\D)\to \bfL^2(\D)$.
Let $\bfw^\ast\in \bfL^2(\D)\setminus\{\boldsymbol 0\}$ be an adjoint eigenfunction, i.e.,
\begin{equation}
  \label{eq:aux-phi-prime-adjoint}
  \bigl(\tilde{\Phi}_{\sigma}'(\bfu)[\bfv],\bfw^\ast\bigr)_{\bfL^2(\D)}
  = \mu\,(\bfv,\bfw^\ast)_{\bfL^2(\D)}
  \qquad\mbox{for all }\,\bfv\in \bfL^2(\D).
\end{equation}
Then $\bfw^\ast\in \bfH^1_0(\D)$ and $\bfw^\ast\in T_{\bfu}\S\cap T_{\ci \bfu}\S$.
Moreover,
$$
  \tilde{\Phi}_{\sigma}'(\bfu)[\bfu] = 0
  \quad\text{and}\quad
  \tilde{\Phi}_{\sigma}'(\bfu)[\ci\bfu] = 0.
$$
\end{lemma}

\begin{proof}
By \eqref{eq:phi-aux-prime}, \eqref{eq:psi_prime}, and \eqref{eq:phi_prime}, the operator $\tilde{\Phi}_{\sigma}'(\bfu)$ is obtained by applying the $\bfL^2$--orthogonal projection onto the orthogonal complement of $\operatorname{span}_{\R}\{\bfu,\ci\bfu\}$ to $(\lambda-\sigma)\calJsigma{\bfu}^{-1}\mathcal{I}$. This also defines its compact extension to $\bfL^2(\D)$. Since the projection is self-adjoint, the adjoint of $\tilde{\Phi}_{\sigma}'(\bfu)$ is obtained by applying $(\lambda-\sigma)(\calJsigma{\bfu}^{*})^{-1}\mathcal{I}$ after this projection. By the argument in Lemma~\ref{lem:J-adjoint}, the latter operator maps $\bfL^2(\D)$ continuously into $\bfH_0^1(\D)$. Since $\mu\neq0$, the adjoint eigenrelation implies $\bfw^\ast\in\bfH_0^1(\D)$.

\medskip\noindent
\emph{Step 1: $\tilde{\Phi}_{\sigma}'(\bfu)[\bfu] = 0$ and
$\tilde{\Phi}_{\sigma}'(\bfu)[\ci\bfu]=0$:}\\[0.4em]
From the definition of $\Psi_{\sigma}$ and 
$\calJ{\bfu}\bfu=\lambda\,\mathcal{I}\bfu$, one has $\Psi_{\sigma}'(\bfu)[\bfw]
  = (\lambda-\sigma)\,\calJsigma{\bfu}^{-1}\mathcal{I}\bfw$,  so in particular it holds $\Psi_{\sigma}'(\bfu)[\bfu] = \bfu$ and $\Psi_{\sigma}'(\bfu)[\ci\bfu ]= \ci\bfu$.
Since $(\cdot,\cdot)_{\bfL^2(\D)}$ is real, we obtain
$$
  \bigl(\Psi_{\sigma}'(\bfu)[\bfu],\ci\bfu\bigr)_{\bfL^2(\D)}
  = (\bfu,\ci\bfu)_{\bfL^2(\D)} = 0,
  \qquad
  \bigl(\Psi_{\sigma}'(\bfu)[\ci\bfu],\ci\bfu\bigr)_{\bfL^2(\D)}
  = (\ci\bfu,\ci\bfu)_{\bfL^2(\D)} = 1.
$$
Moreover, from the formula for $\Phi_{\sigma}'(\bfu)$ we have
$$
  \Phi_{\sigma}'(\bfu)[\bfw]
  = (\lambda-\sigma)\,\calJsigma{\bfu}^{-1}\mathcal{I}\bfw
    - \bigl((\lambda-\sigma)\,\calJsigma{\bfu}^{-1}\mathcal{I}\bfw,\bfu\bigr)_{\bfL^2(\D)}\,\bfu,
$$
so
$$
  \Phi_{\sigma}'(\bfu)[\bfu] = \bfu - (\bfu,\bfu)_{\bfL^2(\D)}\,\bfu = \mathbf0,
$$
and using the phase invariance $\Phi_{\sigma}(\exp(\ci\omega)\,\bfu)=\exp(\ci\omega)\bfu$ (cf. \eqref{phase-identity-Phi-sigma}) we get
$$
  \Phi_{\sigma}'(\bfu)[\ci\bfu ]= \ci\bfu.
$$
Now insert $\bfw=\bfu$ into \eqref{eq:phi-aux-prime} to get 
$$
  \tilde{\Phi}_{\sigma}'(\bfu)[\bfu]
  = -\ci\,\bigl(\Psi_{\sigma}'(\bfu)[\bfu],\ci\bfu\bigr)_{\bfL^2(\D)}\,\bfu
    + \Phi_{\sigma}'(\bfu)[\bfu]
  = -\ci\cdot 0\cdot\bfu + \mathbf{0} = \mathbf{0}
$$
and inserting $\bfw=\ci\bfu$ in \eqref{eq:phi-aux-prime} gives
$$
  \tilde{\Phi}_{\sigma}'(\bfu)[\ci\bfu]
  = -\ci\,\bigl(\Psi_{\sigma}'(\bfu)[\ci\bfu],\ci\bfu\bigr)_{\bfL^2(\D)}\,\bfu
    + \Phi_{\sigma}'(\bfu)[\ci\bfu]
  = -\ci\cdot 1\cdot\bfu + \ci\bfu
  = \mathbf0.
$$

\medskip\noindent
\emph{Step 2: Orthogonality to $\bfu$ and $\ci\bfu$.}\\[0.4em]
With $\bfv=\bfu$ in \eqref{eq:aux-phi-prime-adjoint}, we obtain
$$
  \bigl(\tilde{\Phi}_{\sigma}'(\bfu)[\bfu], \bfw^\ast \bigr)_{\bfL^2(\D)}
  = \mu\,(\bfu,\bfw^\ast)_{\bfL^2(\D)}.
$$
The left-hand side vanishes by Step~1, so we get $\mu\,(\bfu,\bfw^\ast)_{\bfL^2(\D)} = 0$. 
Since $\mu\neq 0$, it follows that
$(\bfu,\bfw^\ast)_{\bfL^2(\D)} = 0$, i.e.,\ $\bfw^\ast\in T_{\bfu}\S$.

Similarly, choosing $\bfv=\ci\bfu$ in \eqref{eq:aux-phi-prime-adjoint} and
using $\tilde{\Phi}_{\sigma}'(\bfu)[\ci\bfu]=\mathbf0$ from Step~1 gives
$$
  0 = \mu\,(\ci\bfu,\bfw^\ast)_{\bfL^2(\D)},
$$
so again $\mu\neq 0$ implies
$(\ci\bfu,\bfw^\ast)_{\bfL^2(\D)}=0$, i.e.,\ $\bfw^\ast\in T_{\ci\bfu}\S$.

We conclude $\bfw^\ast\in T_{\bfu}\S\cap T_{\ci\bfu}\S$ and the lemma is proved.
\end{proof}

We are now ready to establish a relation between the spectra of the $J$-operator
and $\tilde{\Phi}_{\sigma}'(\bfu)$.

\begin{lemma}\label{lem:spectral-relation}
Assume \ref{A1}--\ref{A3} and \ref{A5}, and let $(\bfu,\lambda)\in\S\times\R$ be a ground-state eigenpair of~\eqref{eq:nonlinear-eigenvalue-problem}, where $\bfu$ is quasi-isolated in the sense of \ref{A4}. Let $\sigma\notin\Lambda(\bfu)$, so that $\calJsigma{\bfu}$ is invertible. Then the following holds:
\begin{itemize}
\item[(i)] If $\mu\in\R$, $\mu\neq\lambda$, is an eigenvalue of $\calJ{\bfu}$,
then
$$
  \tilde\mu := \frac{\lambda-\sigma}{\mu-\sigma}\in\R
$$
is an eigenvalue of the compact operator
$\tilde{\Phi}_{\sigma}'(\bfu): \bfL^2(\D)\to \bfL^2(\D)$.

\item[(ii)] Conversely, if $\tilde\mu\in\R$, $\tilde\mu\neq 0$, is an eigenvalue
of $\tilde{\Phi}_{\sigma}'(\bfu)$, then there exists a real eigenvalue
$\mu\neq\lambda$ of $\calJ{\bfu}$ such that
$$
  \tilde\mu = \frac{\lambda-\sigma}{\mu-\sigma}.
$$
\end{itemize}
Moreover, the eigenvalue $\lambda$ of $\calJ{\bfu}$ corresponds to the
eigenvalue $0$ of $\tilde{\Phi}_{\sigma}'(\bfu)$ in the sense that
$\bfu,\ci\bfu\in\ker\tilde{\Phi}_{\sigma}'(\bfu)$.
\end{lemma}

\begin{proof}
By Lemma~\ref{lem:aux-phi-prime-adjoint}, we already know \,\,$\tilde{\Phi}_{\sigma}'(\bfu)[\bfu] = \tilde{\Phi}_{\sigma}'(\bfu)[\ci\bfu]  = \mathbf{0}$,\,\, so $\lambda$ corresponds to the eigenvalue $0$ of $\tilde{\Phi}_{\sigma}'(\bfu)$. To establish the remaining spectral correspondence, we split
\begin{align*}
\bfH_0^1(\D)=Z\oplus W,\qquad
\mbox{for}\quad
Z := \operatorname{span}_{\R}\{\bfu,\ci\bfu\}
\quad\mbox{and}\quad W := T_{\bfu}\S\cap T_{\ci\bfu}\S.
\end{align*}
Since $\bfu\in\S$ is normalized and $(\bfu,\ci\bfu)_{\bfL^2(\D)}=0$, the decomposition is $\bfL^2$--orthogonal. Let $P_W:\bfH_0^1(\D)\to W$ denote the projection associated with this $\bfL^2$--orthogonal decomposition, i.e.,
\begin{align*}
P_W\bfv \,:=\, \bfv -(\bfv,\bfu)_{\bfL^2(\D)}\bfu -(\bfv,\ci\bfu)_{\bfL^2(\D)}\ci\bfu.
\end{align*}
Furthermore, we introduce the compact operator
\begin{align*}
T_\sigma \,:=\, (\lambda-\sigma)\calJsigma{\bfu}^{-1}\mathcal{I}.
\end{align*}
Recall from \eqref{eq:psi_prime}, \eqref{eq:phi_prime} and \eqref{eq:phi-aux-prime} that
\begin{align}
\label{eq:tildePhisigma-der-new}
\tilde{\Phi}_{\sigma}'(\bfu)[\bfv]
&= T_\sigma\bfv -(T_\sigma\bfv,\bfu)_{\bfL^2(\D)}\bfu -(T_\sigma\bfv,\ci\bfu)_{\bfL^2(\D)}\ci\bfu 
\,=\, P_W (T_\sigma\bfv)
\end{align}
for all $\bfv\in\bfH_0^1(\D)$. Since $T_\sigma:\bfL^2(\D)\to\bfH_0^1(\D)$ is bounded, the right-hand side of \eqref{eq:tildePhisigma-der-new} also defines a bounded extension $\tilde{\Phi}_{\sigma}'(\bfu):\bfL^2(\D)\to\bfH_0^1(\D)$. By the compact embedding $\bfH_0^1(\D)\hookrightarrow\bfL^2(\D)$, this extension is compact as an operator on $\bfL^2(\D)$. Moreover, every eigenfunction corresponding to a nonzero eigenvalue of this extension belongs to $\bfH_0^1(\D)$, so its nonzero spectrum agrees with that of the $\bfH_0^1(\D)$-realization.

Since $\calJ{\bfu}\bfu =\lambda\mathcal{I}\bfu$ and $\calJ{\bfu}[\ci\bfu] =\lambda\mathcal{I}[\ci\bfu]$, we obtain
\begin{align*}
T_\sigma\bfu=\bfu, \qquad T_\sigma[\ci\bfu]=\ci\bfu.
\end{align*}
In particular, $T_\sigma$ acts as the identity on $Z$.

\medskip\noindent
\emph{(i) From $\calJ{\bfu}$ to
$\tilde{\Phi}_{\sigma}'(\bfu)$.}
Let $\mu\neq\lambda$ be a real eigenvalue of $\calJ{\bfu}$, so there exists
$\bfz\neq\mathbf0$ with
\begin{align*}
\calJ{\bfu}\bfz=\mu\,\mathcal I\bfz.
\end{align*}
By the shift-invert correspondence, we have $\displaystyle T_\sigma\bfz
= \tfrac{\lambda-\sigma}{\mu-\sigma}\bfz$.
Set $\displaystyle \tilde{\mu}:= \tfrac{\lambda-\sigma}{\mu-\sigma}$ and consider $\bfw:=P_W\bfz$. We first observe that $\bfw\neq\mathbf0$. Indeed, otherwise $\bfz\in Z$ and since $T_\sigma$ acts as the identity on $Z$, we would have $\tilde\mu=1$, which contradicts $\mu\neq\lambda$.

Now with $\bfw\neq\mathbf0$ and using the decomposition $\bfz=(\bfz-P_W\bfz)+\bfw$ and that $T_\sigma$ is the identity on $Z$, we obtain
\begin{align*}
\tilde{\Phi}_{\sigma}'(\bfu)[\bfw] \,\overset{\eqref{eq:tildePhisigma-der-new}}{=}\, P_W (T_\sigma\bfw)
\,=\, P_W (T_\sigma\bfz) \,=\, \tilde{\mu} \,P_W\bfz
\,=\, \tilde{\mu}\,\bfw.
\end{align*}
Hence $\tilde\mu$ is an eigenvalue of $\tilde{\Phi}_{\sigma}'(\bfu)$.

\medskip\noindent
\emph{(ii) From $\tilde{\Phi}_{\sigma}'(\bfu)$ to $\calJ{\bfu}$.}
Let now $\tilde\mu\in\R\setminus\{0\}$ be an eigenvalue of $\tilde{\Phi}_{\sigma}'(\bfu)$, with corresponding eigenfunction $\bfv\neq\mathbf0$. Since $\tilde{\Phi}_{\sigma}'(\bfu)$ maps $\bfL^2(\D)$ into $\bfH_0^1(\D)$, we have $\bfv\in\bfH_0^1(\D)$. Furthermore, since its range is contained in $W$ (by \eqref{eq:tildePhisigma-der-new}) and $\tilde\mu\neq0$, we obtain $\bfv\in W$. Using $\tilde{\Phi}_{\sigma}'(\bfu)=P_W T_\sigma$, we can write
\begin{align}
\label{eq:Tsigma-mu-gamma}
T_\sigma\bfv = \tilde\mu\,\bfv+\boldsymbol{\gamma},
\qquad
\mbox{for some } \boldsymbol{\gamma}\in Z.
\end{align}
We first show that $\tilde\mu\neq1$. Assume that the opposite holds, i.e., that $\tilde\mu=1$. Then
\begin{align*}
(T_\sigma-\mathrm{Id})\bfv
=
\boldsymbol{\gamma}.
\end{align*}
Applying $\calJsigma{\bfu}$ and using 
$\calJsigma{\bfu}\boldsymbol{\gamma} =(\lambda-\sigma)\mathcal{I}\boldsymbol{\gamma}$ (since $\calJ{\bfu}\bfu =\lambda\mathcal{I}\bfu$ and $\calJ{\bfu}[\ci\bfu] =\lambda\mathcal{I}[\ci\bfu]$),
we obtain
\begin{align*}
(\calJ{\bfu}-\lambda\mathcal I)\bfv
=
-(\lambda-\sigma)\mathcal I\boldsymbol{\gamma}.
\end{align*}
Testing this identity with $\bfv\in W$ and using \eqref{identity-Ju-Eprimeprime} yields
\begin{align*}
\langle (E''(\bfu)-\lambda\mathcal I)\bfv,\bfv \rangle
\,=\, -(\lambda-\sigma) (\boldsymbol{\gamma},\bfv)_{\bfL^2(\D)} \,=\,0,
\end{align*}
since $\boldsymbol{\gamma}\in Z$ and $\bfv\in W$. This contradicts the quasi-isolation assumption \ref{A4}, because $\bfv\neq\mathbf0$. Hence $\tilde\mu\neq1$.

Now define $\displaystyle\bfw := \bfv+\tfrac{1}{\tilde\mu-1}\boldsymbol{\gamma}$. 
Since $\bfv\in W\setminus\{\mathbf0\}$
and $\boldsymbol{\gamma}\in Z$, we have
$\bfw\neq\mathbf0$.
Moreover, since $T_\sigma$ acts as the
identity on $Z$, it follows that
\begin{align*}
T_\sigma\bfw
\,=\, T_\sigma\bfv \,+\, \tfrac{1}{\tilde\mu-1} \, T_\sigma\boldsymbol{\gamma} 
\,\overset{\eqref{eq:Tsigma-mu-gamma}}{=}\, 
\tilde\mu\,\bfv+\boldsymbol{\gamma} \,+\,\tfrac{1}{\tilde\mu-1}\, \boldsymbol{\gamma}
\,=\, \tilde\mu\,\bfw.
\end{align*}
Consequently,
\begin{align*}
\calJ{\bfu}\bfw = \mu\,\mathcal I\bfw,
\qquad
\mu:= \sigma+\tfrac{\lambda-\sigma}{\tilde\mu},
\end{align*}
hence, $\mu$ is a real eigenvalue of
$\calJ{\bfu}$.
Since $\tilde\mu\neq1$, we have
$\mu\neq\lambda$, and by construction
\begin{align*}
\tilde\mu = \tfrac{\lambda-\sigma}{\mu-\sigma}.
\end{align*}
Finally, we verify that this relation also describes the nonzero spectrum of the complexification of $\tilde{\Phi}_{\sigma}'(\bfu)$. 
As established above, the extension of $\tilde{\Phi}_{\sigma}'(\bfu)$ to $\bfL^2(\D)$ is compact. Hence its complexification is also compact and every nonzero spectral value of the complexified operator is an eigenvalue. The preceding decomposition and calculations remain valid after complexification. In particular, the eigenvalue $1$ cannot occur in the complexified problem, since its real and imaginary parts would otherwise yield a nonzero real eigenfunction with eigenvalue $1$, which we have already excluded. For any other nonzero complex eigenvalue $\tilde\mu$, the construction above yields 
\begin{align*}
\mu=\sigma+\tfrac{\lambda-\sigma}{\tilde\mu}
\end{align*}
as an eigenvalue of the complexified generalized $J$-eigenvalue problem. By Lemma~\ref{lem:J-real-spectrum}, this eigenvalue is real, and consequently $\tilde\mu$ must be real as well.

Finally, we verify that the kernel contains no additional directions. Since $T_\sigma$ is injective and acts as the identity on $Z$, the identity $P_W T_\sigma\bfv=\mathbf0$ implies $T_\sigma\bfv\in Z$ and hence
\begin{align*}
T_\sigma(\bfv-T_\sigma\bfv)=\mathbf0.
\end{align*}
Consequently, $\bfv=T_\sigma\bfv\in Z$, proving that $\ker\tilde{\Phi}_{\sigma}'(\bfu)=Z$.

This proves the claimed correspondence between the nonzero spectrum of $\tilde{\Phi}_{\sigma}'(\bfu)$ and the eigenvalues of $\calJ{\bfu}$ different from $\lambda$.
\end{proof}

As a direct consequence of Lemma~\ref{lem:spectral-relation}, the nonzero spectrum of $\tilde{\Phi}_{\sigma}'(\bfu)$ for $\sigma \not\in\Lambda(\bfu)$ consists precisely of
$$
  \tilde\mu
  = \frac{\lambda-\sigma}{\mu-\sigma},
  \qquad
  \mu\in\Lambda(\bfu) \setminus \{ \lambda\}.
$$
Moreover, $\ker\tilde{\Phi}_{\sigma}'(\bfu) =\operatorname{span}_{\R}\{\bfu,\ci\bfu\}$, so $0$ is an eigenvalue of geometric multiplicity two. Since $\tilde{\Phi}_{\sigma}'(\bfu)$ is compact, its spectral radius is
$$
  \rho(\tilde{\Phi}_{\sigma}'(\bfu))
  = \max_{\mu\in\Lambda(\bfu),\,\mu\neq\lambda}
     \left|\frac{\lambda-\sigma}{\mu-\sigma}\right|
  = \frac{|\lambda-\sigma|}
         {\min\limits_{\mu\in\Lambda(\bfu),\,\mu\neq\lambda}|\mu-\sigma|}.
$$
We now formulate Ostrowski's theorem for the auxiliary iteration.

\begin{lemma}\label{lem:convergence-aux}
Assume \ref{A1}--\ref{A3} and \ref{A5}. Let $(\bfu,\lambda)\in\S\times\R$ be a ground-state eigenpair that is quasi-isolated in the sense of \ref{A4}, and let
$\tilde{\Phi}_{\sigma}:\mathcal{U}_\sigma\to\bfH_0^1(\D)$ be the auxiliary mapping
with $\sigma\notin\Lambda(\bfu)$. Then:
\begin{enumerate}[label={(\roman*)}]
\item The Fr\'echet derivative $\tilde{\Phi}_{\sigma}'(\bfu):\bfH_0^1(\D)\to\bfH_0^1(\D)$
is compact and its spectral radius is
$$
  \rho := \rho(\tilde{\Phi}_{\sigma}'(\bfu))= \frac{|\lambda-\sigma|}{\min\limits_{\mu\in\Lambda(\bfu),\,\mu\neq\lambda}|\mu-\sigma|},
$$
with $0$ being an eigenvalue of geometric multiplicity two corresponding to the eigenfunctions $\bfu \text{ and }\ci\bfu$.
\item
If $\rho<1$, i.e., if $\sigma$ is selected such that
$$
  |\lambda-\sigma|   < \min_{\mu\in\Lambda(\bfu),\,\mu\neq\lambda}|\mu-\sigma|,
$$
then there exists a neighbourhood $U$ of $\bfu$ in $\bfH_0^1(\D)$ such that the
fixed-point iteration
$$
  \tilde\bfu^{n+1} = \tilde{\Phi}_{\sigma}(\tilde\bfu^n)
$$
converges strongly to $\bfu$ for all $\tilde\bfu^0\in U$.
\item If $\rho<1$, then for every $\varepsilon>0$ there exist a neighbourhood $U_{\varepsilon}\subset\mathcal{U}_\sigma$ of $\bfu$ and a constant $C_\varepsilon>0$ such that for all
$\tilde\bfu^0\in U_\varepsilon$ and all $n\ge1$,
$$
  \|\tilde\bfu^{n}-\bfu\|_{\bfH^1(\D)} \;\le\;  C_\varepsilon \,(\rho+\varepsilon)^{n}\, \|\tilde\bfu^{0}-\bfu\|_{\bfH^1(\D)}.
$$
\end{enumerate}
\end{lemma}

\begin{proof}
(i): Compactness follows from Lemma~\ref{lem:compact-tildePhi}.  
Lemma~\ref{lem:spectral-relation} shows that the nonzero spectrum of the complexification of $\tilde{\Phi}'_{\sigma}(\bfu)$ consists precisely of the values $(\lambda-\sigma)/(\mu-\sigma)$
with $\mu\in\Lambda(\bfu)\setminus\{\lambda\}$. As established above, we also have
$\ker\tilde{\Phi}_{\sigma}'(\bfu)=\operatorname{span}_{\R}\{\bfu,\ci\bfu\}$.
Hence, the spectral-radius formula follows directly.\\[0.4em]
(ii): Since $\tilde{\Phi}_{\sigma}'(\bfu)$ is compact,
Ostrowski's theorem (Lemma \ref{lemma:ostrowski}) applies to the fixed point $\bfu$:  
if the spectral radius satisfies $\rho<1$, then the fixed point is locally
attractive.  
Thus, there exists a neighbourhood $U$ of $\bfu$ such that the iteration
converges strongly in $\bfH_0^1(\D)$.\\[0.4em]
(iii): Again, follows by Ostrowski's theorem (Lemma \ref{lemma:ostrowski}).
\end{proof}

\begin{proof}[Proof of Theorem~\ref{theo:local}]
Since $\sigma$ is chosen such that
\begin{align*}
  \rho
  =\frac{|\lambda-\sigma|}
         {\min\limits_{\mu\in\Lambda(\bfu),\,\mu\neq\lambda}|\mu-\sigma|}
  <1,
\end{align*}
Lemma~\ref{lem:convergence-aux}(i)-(iii) apply.  
Thus, for every $\varepsilon>0$ there exist $U_\varepsilon$ and $C_\varepsilon>0$
such that the auxiliary iteration
$$
  \tilde\bfu^{n+1}=\tilde\Phi_\sigma(\tilde\bfu^n),
  \qquad \tilde\bfu^0=\bfu^0,
$$
satisfies
\begin{equation}
  \|\tilde\bfu^{n}-\bfu\|_{\bfH^1(\D)}
  \;\le\;
  C_\varepsilon(\rho+\varepsilon)^n
  \|\bfu^0-\bfu\|_{\bfH^1(\D)}.
  \label{eq:aux-decay}
\end{equation}
By Lemma~\ref{lem:aux-it}, for all $n\ge1$ the auxiliary iteration is linked to the original iteration by  
\begin{align*}
  \tilde\bfu^n=\exp(\ci \omega_n) \,\bfu^n,
  \qquad \mbox{for some }  \omega_n\in \R.
\end{align*}
Hence
\begin{align*}
  \dist_{S^1}^{\bfH^1}(\bfu,\bfu^n)
  = \inf_{\omega \in \R}\| \exp(\ci \omega) \, \bfu^n-\bfu\|_{\bfH^1}
  \le \|\tilde\bfu^n-\bfu\|_{\bfH^1}.
\end{align*}
Together with \eqref{eq:aux-decay} this yields
\begin{align*}
  \dist_{S^1}^{\bfH^1}(\bfu,\bfu^n)
  \le
  C_\varepsilon(\rho+\varepsilon)^n
  \|\bfu^0-\bfu\|_{\bfH^1(\D)},
\end{align*}
which is the desired local linear convergence up to phase.

\medskip
Convergence of the densities follows because $|\tilde\bfu^n|=|\bfu^n|$, hence 
$$
  \||\bfu^n|^2-|\bfu|^2\|_{\bfL^2}
  =\||\tilde\bfu^n|^2-|\bfu|^2\|_{\bfL^2}
  \le C\,\|\tilde\bfu^n-\bfu\|_{\bfH^1},
$$
where the inequality follows from the Sobolev embedding $\bfH^1(\D)\hookrightarrow\bfL^4(\D)$ for $d=2,3$ and the uniform boundedness of $(\tilde\bfu^n)_{n\in \N}$ in $\bfH^1(\D)$, which holds for initial values in a sufficiently small neighbourhood of $\bfu$ by \eqref{eq:aux-decay}.  
Applying \eqref{eq:aux-decay} again gives
\begin{align*}
  \||\bfu^n|^2-|\bfu|^2\|_{\bfL^2(\D)}
  \le
  C_\varepsilon(\rho+\varepsilon)^n
  \|\bfu^0-\bfu\|_{\bfH^1(\D)}.
\end{align*}
This completes the proof.
\end{proof}

\section{Application to spin--orbit coupled quantum gases}
\label{section:SO-application}
We now apply the abstract theory to ground states of spin--orbit coupled pseudo-spin-$1/2$ Bose--Einstein condensates. We verify the structural assumptions required for the $J$-method, examine the quasi-isolation condition numerically, and illustrate the convergence of the method in two numerical experiments.

\subsection{The spin--orbit coupled Gross--Pitaevskii model}

When bosonic gases are cooled down to temperatures near absolute zero, a fascinating exotic state of matter can be observed: the \emph{Bose--Einstein condensate}, cf.~\cite{Bos24,Ein24,PiS03}.
In this state, atoms collectively behave as a single quantum entity, allowing for remarkable physical phenomena such as \emph{superfluidity} and \emph{supersolidity}; see~\cite{Li17}.
Supersolids combine crystalline order with superfluidity, i.e.,\ frictionless flow, and can exhibit complex structures in the ground-state density.

Here we consider two-component \emph{spin--orbit coupled Bose--Einstein condensates} (SO-coupled BECs), as studied in~\cite{Stanescu-et-al-08,Wang-et-al-10}.
Spin--orbit coupling occurs when an atom's spin, i.e.,\ its intrinsic angular momentum, is coupled to its orbital motion in space.
This effect is typically engineered using Raman lasers~\cite{SO-coupled-BEC-experiment}.
The classical motivation for studying Raman-induced spin--orbit coupling is to understand phenomena such as topological phases, supersolidity, and unconventional superfluidity within a tunable quantum-gas platform that can serve as a quantum simulator for solid-state systems~\cite{ZhaiSOcoupledBECs2015}.
The combination of spin--orbit coupling and nonlinear particle interactions leads to complex energy landscapes with potentially numerous local minima, making ground-state computations challenging.

For $N=2$, following the setting established in~\cite{BC15}, we consider the energy $E:\bfH^1_0(\D)\to\mathbb{R}$,
\begin{align}
\label{eq:SO-energy}
E(\bfu)
&:= \frac{1}{2}\int_{\D}\Bigg[
\underbrace{\tfrac{1}{2}(|\nabla u_1|^2 + |\nabla u_2|^2)}_{\text{kinetic}}
+ \underbrace{V_1(x)|u_1|^2 + V_2(x)|u_2|^2}_{\text{trap}}
+ \underbrace{\frac{\delta}{2}(|u_1|^2 - |u_2|^2)
+ \Omega\,\Re(u_1\bar u_2)}_{\text{spin coupling}} \\
\nonumber
&\quad
+\, \underbrace{\ci k_0 (\bar u_1\,\partial_{x_1} u_1 - \bar u_2\,\partial_{x_1} u_2)}_{\text{spin--orbit term}}
+ \underbrace{\frac{\beta_{11}}{2}|u_1|^4
+ \frac{\beta_{22}}{2}|u_2|^4
+ \beta_{12}|u_1|^2|u_2|^2}_{\text{interaction}}
\Bigg]\,\dx.
\end{align}
Here, $\bfu=(u_1,u_2)$ describes the quantum states of the two components on $\D\subset\mathbb{R}^d$, $d=2,3$.
The external potentials $V_j\in L^\infty(\D,\mathbb{R})$ confine the atoms, $\delta$ is the detuning associated with the Raman transition, $\Omega$ is the Rabi coupling (frequency), and $k_0$ is the wave number of the Raman lasers, determining the spin--orbit coupling strength.
The parameters $\beta_{11},\beta_{22}$ and $\beta_{12}$ describe intra- and inter-component interactions, respectively.
The integral in~\eqref{eq:SO-energy} is real by integration by parts and the homogeneous Dirichlet boundary conditions. Ground states minimize $E$ over the sphere $\mathbb{S}$ defined in Section~\ref{subsection-first-second-order-conditions}; they describe the most stable configurations of the condensate.
By the first-order optimality condition derived in Section~\ref{subsection-first-second-order-conditions}, every ground state $\bfu\in\mathbb{S}$ satisfies $E'(\bfu)=\lambda\mathcal{I}\bfu$ for some $\lambda\in\mathbb{R}$. For the SO-coupled energy~\eqref{eq:SO-energy}, this condition becomes the following nonlinear eigenvalue problem, formally written in strong form as
\begin{equation}\label{so-coupling-ev-problem-strong}
\begin{aligned}
\lambda u_1 &= \left[-\tfrac{1}{2}\Delta + V_1(x) + \ci k_0 \partial_{x_1} + \tfrac{\delta}{2}
+ \big(\beta_{11}|u_1|^2 + \beta_{12}|u_2|^2\big)\right]u_1 + \tfrac{\Omega}{2}u_2,\\[0.5em]
\lambda u_2 &= \left[-\tfrac{1}{2}\Delta + V_2(x) - \ci k_0 \partial_{x_1} - \tfrac{\delta}{2}
+ \big(\beta_{12}|u_1|^2 + \beta_{22}|u_2|^2\big)\right]u_2 + \tfrac{\Omega}{2}u_1.
\end{aligned}
\end{equation}
Existing approaches to computing ground states of SO-coupled condensates are primarily based on gradient flows and Riemannian optimization; see~\cite{BC15,HHP25}. Although effective in many regimes with guaranteed energy decay, such methods can converge slowly when the energy landscape becomes nearly flat, as may occur for strong spin--orbit coupling and the associated striped density patterns (large $k_0$).  This motivates applying the shifted $J$-method analyzed
in the preceding sections to the SO-coupled problem. The numerical experiments below compare it with the density-frozen $A$-iteration and illustrate the substantial acceleration achieved by the Jacobian-based approach, particularly for large $k_0$. 

To verify that the relevant assumption for our convergence theory are fulfilled, we consider
\begin{align}
\label{eq:SO-parameter-assumptions}
V_j(x)\geq \frac{|\delta|+|\Omega|+2k_0^2}{2}\quad(j=1,2),
\qquad \beta_{11},\beta_{12},\beta_{22}\geq0.
\end{align}
These conditions ensure coercivity and weak lower semicontinuity of $E$ on $\mathbb{S}$. Indeed, the spin coupling is bounded below by $-(|\delta|+|\Omega|)(|u_1|^2+|u_2|^2)/2$, and Young's inequality gives
\begin{align*}
\Re\big(\ci k_0(\bar u_1\partial_{x_1}u_1-\bar u_2\partial_{x_1}u_2)\big)
\geq -k_0^2\sum_{j=1}^2|u_j|^2
-\frac{1}{4}\sum_{j=1}^2|\partial_{x_1}u_j|^2.
\end{align*}
Thus $E(\bfu)\geq\frac{1}{8}\|\bfu\|_{\bfH^1_0(\D)}^2$ for the gradient norm on $\bfH^1_0(\D)$. The remaining terms are weakly continuous along bounded weakly convergent sequences by the compact Sobolev embeddings for $d=2,3$. Moreover, $E$ is three times continuously real Fr\'echet differentiable and invariant under global phase rotations. Hence the model satisfies the abstract assumptions~\ref{A1}--\ref{A3}.

As in~\cite{HHP25}, we construct the operator $\calA{\bfu}$ by freezing the nonlinear density coefficients in~\eqref{so-coupling-ev-problem-strong} and dividing them by $\|\bfu\|_{\bfL^2(\D)}^2$ to ensure real-scaling invariance. For $\bfu\neq\mathbf{0}$, we hence obtain

\begin{align}
\label{def:SO-Au} \langle\calA{\bfu}\bfv,\bfw\rangle
&:=\Re\int_\D\Bigg[
\sum_{j=1}^2\Big(
\frac{1}{2}\nabla v_j\cdot\nabla\conjw{j}
+V_jv_j\conjw{j}\Big)
+\frac{\delta}{2}(v_1\conjw{1}-v_2\conjw{2})
+\frac{\Omega}{2}(v_1\conjw{2}+v_2\conjw{1})\\
\nonumber&\quad+\ci k_0(\conjw{1}\partial_{x_1}v_1-\conjw{2}\partial_{x_1}v_2)
+\frac{
\sum_{j=1}^2\beta_{jj}|u_j|^2v_j\conjw{j}
+\beta_{12}(|u_2|^2v_1\conjw{1}+|u_1|^2v_2\conjw{2})
}{\|\bfu\|_{\bfL^2(\D)}^2}
\Bigg]\,\dx.
\end{align}
The representation differs from the integral construction~\eqref{def-Lu}, but satisfies~\eqref{energy-A-identity} and~\eqref{scaling-invariance-A}. In fact, its product $\calA{\bfu}\bfu$ agrees with~\eqref{normalized-energy-A}, so the resulting $J$-operator is precisely the one studied in the abstract analysis.

\begin{lemma}[Verification of the G\r{a}rding assumption \ref{A5}]
\label{lem:SO-Garding}
Under~\eqref{eq:SO-parameter-assumptions}, the operator~\eqref{def:SO-Au} is bounded, symmetric and coercive for each $\bfu\neq\mathbf{0}$, with
\begin{align}
\label{eq:SO-A-coercive}
\langle\calA{\bfu}\bfv,\bfv\rangle
\geq\tfrac{1}{4}\|\bfv\|_{\bfH^1_0(\D)}^2
\qquad\text{for all }\bfv\in\bfH^1_0(\D).
\end{align}
Moreover, for every $\bfu\in\mathbb{S}$ there exists $C(\bfu)>0$ such that
\begin{align}
\label{eq:SO-J-garding}
\langle\calJ{\bfu}\bfv,\bfv\rangle
\geq\tfrac{1}{8}\|\bfv\|_{\bfH^1_0(\D)}^2
-C(\bfu)\|\bfv\|_{\bfL^2(\D)}^2
\qquad\text{for all }\bfv\in\bfH^1_0(\D).
\end{align}
In particular, assumption~\ref{A5} is satisfied at every ground state.
\end{lemma}

\begin{proof}
Estimate~\eqref{eq:SO-A-coercive} follows from the preceding bounds for the linear terms and the nonnegativity of the interaction coefficients; cf.~\cite[Conclusion~4.3]{HHP25}. For the second assertion, recall the tangent-Hessian representation~\eqref{eq:J}. At $\bfu\in\mathbb{S}$, we have
\begin{align*}
\langle E''(\bfu)\bfv,\bfv\rangle
&=\langle\calA{\bfu}\bfv,\bfv\rangle
+2\sum_{j=1}^2\beta_{jj}\int_\D\big(\Re(u_j\bar v_j)\big)^2\,\dx +4\,\beta_{12}\int_\D\Re(u_1\bar v_1)\Re(u_2\bar v_2)\,\dx.
\end{align*}
The diagonal contributions are nonnegative. By H\"older's inequality and $H^1(\D)\hookrightarrow L^6(\D)$,
\begin{align*}
\left|\int_\D\Re(u_1\bar v_1)\Re(u_2\bar v_2)\,\dx\right|
&\leq \|u_1\|_{L^6(\D)}\|u_2\|_{L^6(\D)}
\|v_1\|_{L^6(\D)}\|v_2\|_{L^2(\D)} 
\leq \varepsilon\|\bfv\|_{\bfH^1_0(\D)}^2
+C_{\varepsilon}(\bfu)\|\bfv\|_{\bfL^2(\D)}^2.
\end{align*}
Finally, the correction term in~\eqref{eq:J} satisfies
\begin{align*}
\left|(\bfu,\bfv)_{\bfL^2(\D)}
\big\langle E'(\bfu)-E''(\bfu)\bfu,\bfv\big\rangle\right|
\leq C(\bfu)\|\bfv\|_{\bfL^2(\D)}^2,
\end{align*}
because $E'(\bfu)-E''(\bfu)\bfu$ contains only cubic interaction terms, which belong to $\bfL^2(\D)$ for $\bfu\in\bfH^1_0(\D)$. Combining these estimates with~\eqref{eq:SO-A-coercive} and choosing $\varepsilon>0$ sufficiently small gives~\eqref{eq:SO-J-garding}.
\end{proof}

The quasi-isolation condition~\ref{A4} is an additional property of the ground state under consideration; it is not implied by the model assumptions~\eqref{eq:SO-parameter-assumptions}. Thus, whenever a ground state satisfies~\ref{A4}, the abstract local convergence in Theorem~\ref{theo:local} applies. We numerically check this condition for the ground states computed below by examining the lowest eigenvalues of the tangential Hessian $E''(\bfu)|_{T_{\bfu}\mathbb{S}}$. For the SO-GPE, global energy decay and convergence results for the damped $A$-method are available in~\cite{HHP25}; we use that method to obtain initial approximations for the locally convergent shifted $J$-method.

\subsection{Numerical experiments}
\label{subsection:SO-numerics}

Of particular interest are regimes with very thin, high-contrast stripes in the density, as occur in the supersolid phase.
Such numerically challenging regimes can arise for large Raman recoil momentum $k_0$ and small but nonzero Rabi coupling $\Omega$, where standard approaches may exhibit slow local convergence.
The $J$-method aims to accelerate convergence in such settings while preserving the structure of an inverse iteration.

All computations were performed in \texttt{Julia} on $\D=(-1,1)^2$ with $\beta_{11}=10$, $\beta_{12}=\beta_{22}=9$, $\Omega=50$, $\delta=0$ and $k_0\in\{10,50\}$. We impose homogeneous Dirichlet boundary conditions and choose the constant potentials $V_j=(|\Omega|+2k_0^2)/2$, which only shift the energy by a constant on $\mathbb{S}$. As in~\cite{HHP25}, we use uniform meshes with $2^8$ subintervals in each coordinate direction and $\mathbb{P}^1$ and $\mathbb{P}^2$ Lagrange finite elements. The $A$-iterations in these FEM spaces are denoted by $A1$ and $A2$, respectively; $J2$ denotes the shifted $J$-method with $\mathbb{P}^2$ elements. The initial functions are the FEM interpolations of
$$
u_2^0(x)\,=\,2\,u_1^0(x)\,=\,(x_1-1)^2(x_2-1)^2 \exp(-\tfrac{\ci}{2}(x_1^2+x_2^2)),
$$
with boundary degrees of freedom set to zero, followed by $\bfL^2$-normalization. We switch from $A2$ to $J2$ when the energy change between consecutive $A2$-iterates falls below $10^{-4}$. In the $A1$--$A2$--$J2$ strategy, the same threshold triggers interpolation of the $A1$ iterate from the $\mathbb{P}^1$ to the $\mathbb{P}^2$ space. The iterations are stopped once $|E(\bfu^n)-E_{\mathrm{ref}}|<10^{-12}$. The $J2$-shift is either fixed or, in the adaptive runs, recomputed during the first two $J2$-steps for $k_0=10$ and the first five for $k_0=50$, and then kept constant to avoid nearly singular systems. To implement the $J$-step, the Woodbury formula reduces the solve to three linear systems instead of forming a dense inverse.

For $k_0=10$, the $A2$-method takes $1871$ steps to reach the prescribed energy tolerance, whereas switching after $124$ steps to $J2$ requires only $6$ further steps with a fixed shift or $3$ with an adaptive shift. The reference values are $E_{\mathrm{ref}}\approx38.214314227545$ and $\lambda\approx78.376222733267$. The energy error in Figure~\ref{fig:energy-differences-comparison} (left) exhibits locally linear decay for fixed shifting and a substantially faster decay during adaptive shifting.

For $k_0=50$, the density of the first component develops many narrow stripes, while the second remains comparatively smooth (Figure~\ref{fig:SO-densities}). This more demanding regime benefits from a coarse-to-fine strategy: $A2$ followed by $J2$ takes $831+12$ steps, whereas $A1$--$A2$--$J2$ takes $288+425+9$ steps, see Figure~\ref{fig:energy-differences-comparison} (right) and Table \ref{table:SO-comparison} (right). The reference values are $E_{\mathrm{ref}}\approx639.782526324267$ and $\lambda\approx1281.531474207$. We also numerically verify the quasi-isolation condition~\ref{A4} for the computed ground states. Indeed, global phase invariance implies that $\ci\bfu$ is a neutral direction of the constrained Hessian, so that $\lambda$ is an eigenvalue of $E''(\bfu)|_{T_{\bfu}\mathbb{S}}$. For $k_0=10$, the two smallest computed eigenvalues are $78.376222733300\approx\lambda$ and $78.636667787518$, with a positive spectral gap of approximately $0.26045$; for $k_0=50$, they are $1281.531474207026\approx\lambda$ and $1281.534588491403$, with a gap of approximately $0.00311$. Hence, in both discretized problems, the phase direction is numerically found to be the only neutral direction, which verifies the quasi-isolation criterion~\ref{A4} at the discrete level (up to numerical accuracy). The much smaller gap for $k_0=50$ is consistent with a flatter energy landscape and the observed slow progress of the $A$-method.

\begin{figure}[tbp]
\centering
\begin{subfigure}[t]{0.49\textwidth}
\centering
\begin{tikzpicture}
\begin{axis}[
  ymode=log,
  restrict x to domain=0:1900,
  xmin=0, xmax=1900,
  xtick={0,500,1000,1500},
  width=7.6cm, height=5.7cm,
  xlabel={Number of iterations},
  ylabel={Difference to reference energy},
  font=\scriptsize,
  label style={font=\scriptsize},
  tick label style={font=\scriptsize},
  legend style={
    at={(0.5,-0.27)},
    anchor=north,
    legend columns=2,
    font=\scriptsize,
    column sep=5pt,
    draw=none,
    legend cell align=left
  },
  every axis plot/.append style={line width=0.95pt},
]
\addplot[solid, gray, thin, restrict x to domain=125:130] file {fig1_4_reference_rate.txt};
\addlegendentry{Linear reference}
\addplot[solid, green(pigment), thick] file {fig1_1_Energy_diff_A2.txt};
\addlegendentry{$A2$}
\addplot[solid, orange, thick] file {fig1_3_Energy_diff_adaptive_A2.txt};
\addlegendentry{$J2$: adaptive shift}
\addplot[dashed, blue, thick] file {fig1_2_Energy_diff_fixed_A2.txt};
\addlegendentry{$J2$: fixed shift}

\addlegendimage{empty legend}
\addlegendentry{}

\addplot[red, only marks, mark=*, mark size=1.25pt, forget plot]
coordinates {(124,0.008583659401807608)};
\coordinate (insetPosition) at (rel axis cs:0.96,0.51);
\end{axis}

\begin{axis}[
  ymode=log,
  at={(insetPosition)},
  anchor=outer south east,
  width=3.8cm, height=2.65cm,
  restrict x to domain=124:130,
  unbounded coords=jump,
  xmin=124, xmax=130,
  xtick={124,127,130},
  ytick={1e-1,1e-5,1e-9,1e-13},
  font=\scriptsize,
  tick label style={font=\scriptsize},
]
\addplot[solid,gray,thin] file {fig1_4_reference_rate.txt};
\addplot[solid,green(pigment),thick] file {fig1_1_Energy_diff_A2.txt};
\addplot[solid,orange,thick] file {fig1_3_Energy_diff_adaptive_A2.txt};
\addplot[dashed,blue,thick] file {fig1_2_Energy_diff_fixed_A2.txt};
\addplot[red,only marks,mark=*,mark size=1pt]
coordinates {(124,0.008583659401807608)};
\end{axis}
\end{tikzpicture}
\vspace{0.1cm}
\caption{$k_0=10$ (with inset).}
\label{fig2_diff_reference_energy_k0_10}
\end{subfigure}\hfill%
\begin{subfigure}[t]{0.49\textwidth}
\centering
\begin{tikzpicture}
\begin{axis}[
  ymode=log,
  restrict x to domain=0:1000,
  xmin=0, xmax=1000,
  xtick={0,250,500,750,1000},
  width=7.6cm, height=5.7cm,
  xlabel={Number of iterations},
  ylabel={Difference to reference energy},
  font=\scriptsize,
  label style={font=\scriptsize},
  tick label style={font=\scriptsize},
  legend style={
    at={(0.5,-0.27)},
    anchor=north,
    legend columns=2,
    font=\scriptsize,
    column sep=5pt,
    draw=none,
    legend cell align=left
  },
  every axis plot/.append style={line width=0.95pt},
]
\addplot[solid,violet,thick] file {fig3_1_Energy_diff_A1_50.txt};
\addlegendentry{$A1$}
\addplot[solid,green(pigment),thick] file {fig3_2_Energy_diff_A2_50.txt};
\addlegendentry{$A2$ after $A1$}
\addplot[solid,orange,thick] file {fig3_3_Energy_diff_J2_50.txt};
\addlegendentry{$J2$ after $A1$--$A2$}
\addplot[dashed,green(pigment),thick] file {fig3_4_Energy_diff_A2J2_A2_50.txt};
\addlegendentry{$A2$ only}
\addplot[dashed,orange,thick] file {fig3_5_Energy_diff_A2J2_J2_50.txt};
\addlegendentry{$J2$ after $A2$}

\addplot[red,only marks,mark=square*,mark size=1.25pt,forget plot]
coordinates {(288,7.736979312055382)};
\addplot[red,only marks,mark=*,mark size=1.25pt,forget plot]
coordinates {(831,0.014853246016173216) (713,0.014531632359876312)};
\end{axis}
\end{tikzpicture}
\vspace{-0.4cm}
\caption{$k_0=50$.}
\label{fig4_diff_k0_50}
\end{subfigure}
\caption{Difference to reference energy for $k_0 \in \{ 10,50\}$. Red markers indicate switches between iteration methods. Solid lines in (b) show the $A1$--$A2$--$J2$ strategy; dashed lines the $A2$--$J2$ strategy.}
\label{fig:energy-differences-comparison}
\end{figure}

\begin{table}[tbp]
\centering
\begin{subtable}[t]{0.48\textwidth}
\centering
\begin{tabularx}{\linewidth}{|X||l|}
\hline
Method & No. of iterations\\
\hline
A2 & 1871\\
A2-J2 (adaptive) & 124 + 3\\
A2-J2 (fixed) & 124 + 6\\
\hline
\end{tabularx}
\caption{$k_0=10$.}
\label{table-energy-final-iteration}
\end{subtable}\hfill
\begin{subtable}[t]{0.48\textwidth}
\centering
\begin{tabularx}{\linewidth}{|X||l|}
\hline
Method & No. of iterations\\
\hline
A2-J2 (adaptive) & 831 + 12\\
A1-A2-J2 (adaptive) & 288 + 425 + 9\\
\hline
\end{tabularx}
\caption{$k_0=50$.}
\label{table_k0_50}
\end{subtable}
\caption{Comparison of iteration counts for the different methods.}
\label{table:SO-comparison}
\end{table}

\begin{figure}[t]
\centering
\begin{minipage}{0.48\textwidth}
\centering
\includegraphics[width=\linewidth]{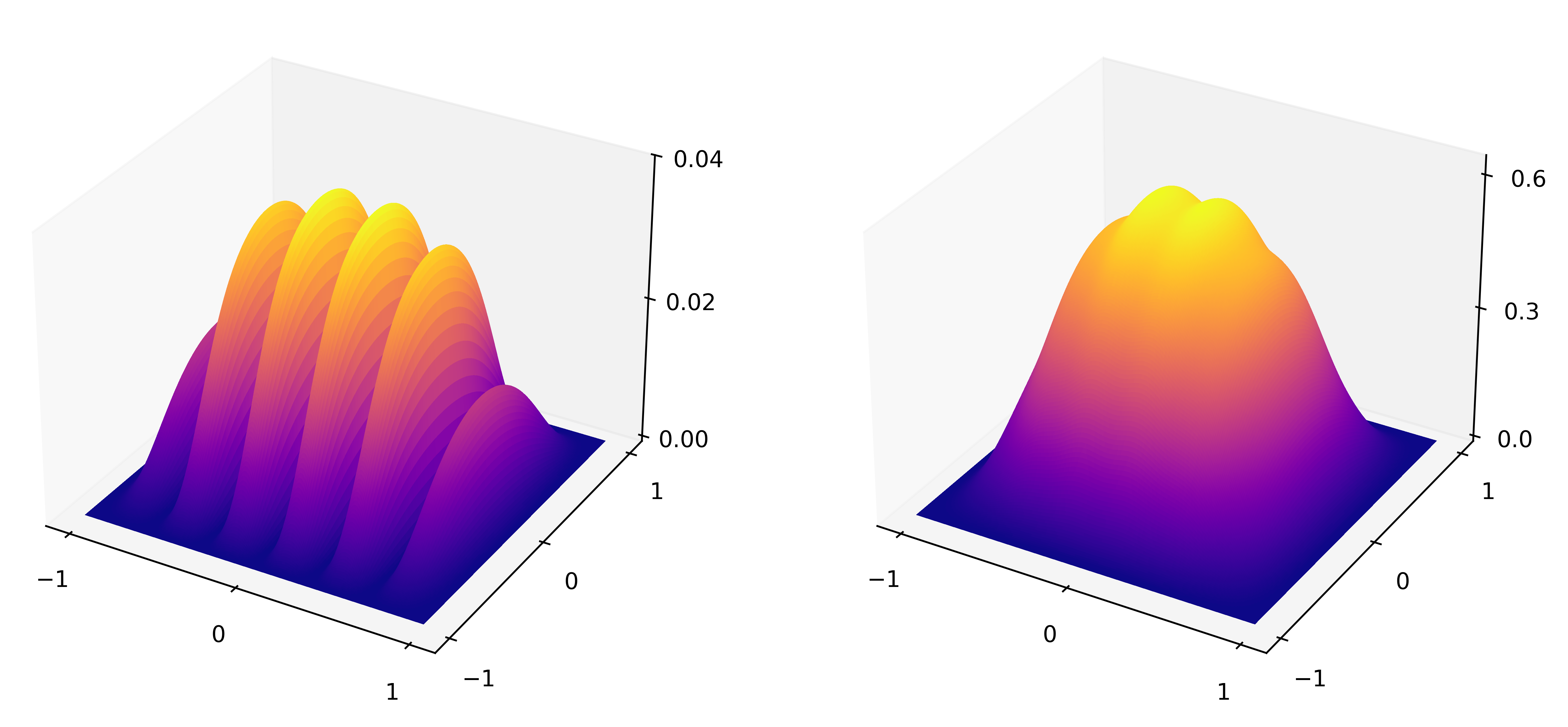}
\end{minipage}\hfill
\begin{minipage}{0.48\textwidth}
\centering
\includegraphics[width=\linewidth]{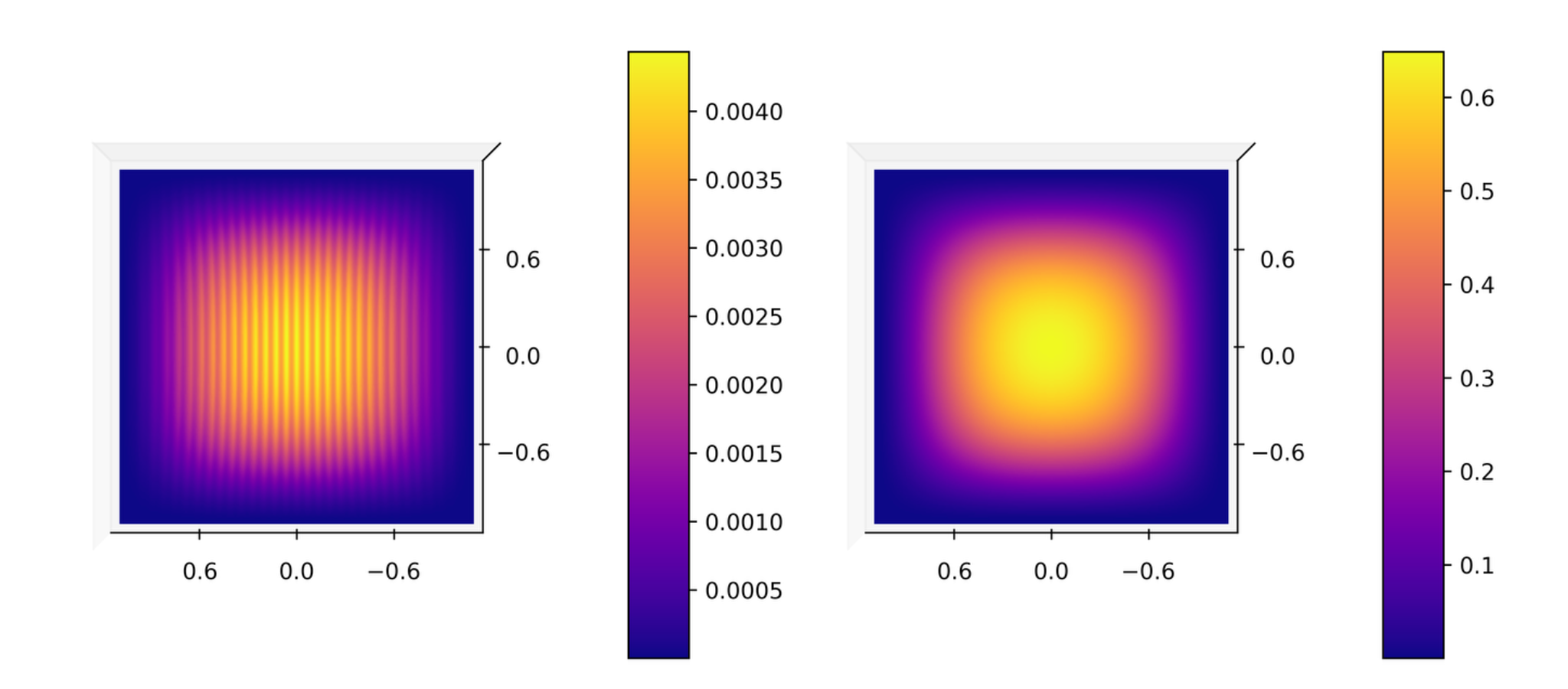}
\end{minipage}
\caption{Ground-state densities $|u_1|^2$ and $|u_2|^2$ at $k_0=10$ (left) and $k_0=50$ (right). The second illustration shows top views cropped from the original density figure such that the extremely thin stripes become visible.}
\label{fig:SO-densities}
\label{fig3_density_k0_10}
\label{fig5_density_k0_50}
\end{figure}

Overall, the experiments illustrate the efficiency of combining the energy-decreasing $A$-method with the locally accelerated $J$-method. Adaptive shifts suggest faster-than-linear convergence in these experiments; a rigorous analysis of adaptive shifting is beyond the scope of the present work.

\end{document}